\documentclass[11pt]{amsart}
\usepackage{color}
\usepackage{amssymb}
\usepackage{verbatim} 
\newtheorem{Thm}[equation]{Theorem}
\newtheorem{Pro}[equation]{Proposition}
\newtheorem{cor}[equation]{Corollary}
\newtheorem{lem}[equation]{Lemma}
\newtheorem{Example}[equation]{\rm Example}
\newtheorem{con}[equation]{\rm Convention}
\newtheorem{deft}[equation]{\rm Definition}
\newtheorem{rem}[equation]{\rm Remark}

\numberwithin{equation}{section}

\address{%
}

\def\a1s{a_1,\cdots, a_s}
\def\a{\alpha}

\def\aa{\mathcal A}

\def\andd{\quad\hbox{and}\quad}

\def\b{\beta}

\def\bb{\mathcal{B}}

\def\bl4{B_{\ell\geq4}}

\def\d{\delta}
\def\D{\Delta}

\def\bbbf{\mathbb{F}}

\def\bbbk{\mathbb{K}}
\def\lam{\lambda}
\def\Lam{\Lambda}

\def\ep{\epsilon}
\def\fm{(\cdot,\cdot)}

\def\pp{\mathcal{P}}

\def\bbbq{\mathbb{Q}}

\def\1k{\frac{1}{k}}
\def\op{\oplus}
\def\ot{\otimes}

\def\la{\langle}
\def\ra{\rangle}

\def\sub{\subseteq}

\def\rcross{R^{\times}}

\def\pf{\noindent{ Proof. }}

\def\fp{\mathfrak{p}}

\def\sfs{\mathfrak{s}}

\def\T{{\mathcal T}}

\def\u{{\mathcal U}}

\def\v{{\mathcal V}}

\def\w{{\mathcal W}}

\def\bbbz{{\mathbb Z}}

\def\1il{1\leq i\leq\ell}

\def\rre{R_{re}}
\def\rim{R_{ns}}
\def\supp{\hbox{\rm supp}}

\begin{document}
\centerline{\bf Locally Finite Root Supersystems}
\vspace{0.7cm}
\centerline{Malihe Yousofzadeh\footnote{ma.yousofzadeh@ipm.ir,\\
\emph{\small Mathematics Subject Classification:} 17B22, 17B65,\\
\emph{\small Key Words:} Locally finite root supersystem, Real root, Nonsingular root.}}

\centerline{\footnotesize
Department of Mathematics, University of Isfahan, Isfahan, Iran,}
\centerline{\footnotesize P.O.Box 81745-163, Phone Number:(98)(313 793 4606).}                

\vspace{0.5cm}
\textbf{Abstract.} We introduce the notion of locally finite root supersystems as a generalization of both locally finite root systems and generalized root systems. We classify irreducible locally finite root supersystems.
\section{Introduction}
In 2001, Neeb and Stumme \cite{NS} studied locally finite split simple Lie algebras, i.e., locally finite   simple Lie algebras containing a maximal abelian ad-diagonalizable subalgebra. They showed that a locally finite split simple Lie algebra
is   a direct limit of finite dimensional split simple Lie algebras and its corresponding root system is a direct limit of irreducible reduced finite root systems. In 2004, Loos and Neher \cite{LN} studied   direct limits of finite root systems and called them locally finite root systems. A locally finite root system $R$ is a locally finite spanning set of a nontrivial  vector space of arbitrary dimension, over a field of characteristic zero, equipped with a symmetric bilinear form which is positive definite on the rational space spanned by $R.$ In particular,  a locally finite root system $R$ dose not contain nonsingular roots. One knows that  in  contrast with a finite root system (as the root system of a   finite dimensional split simple Lie algebra), the root system of a finite dimensional basic classical Lie superalgebra contains nonsingular roots. This gave a motivation to  Serganova \cite{serg} in  1996 to introduce the notion of  generalized root systems as a  generalization of finite root systems.
The main difference between generalized root systems and  finite root
systems is the existence of nonsingular roots.
 Serganova classified irreducible generalized root systems and showed  that  such  root systems are root systems
of contragredient Lie superalgebras  \cite{K}.

%

In this work, we introduce the notion of  locally finite root supersystems. Roughly speaking, a spanning set    $R$ of a  nontrivial vector space over a field $\bbbf$ of characteristic zero, equipped with a nondegenerate symmetric bilinear form, is called a locally finite root supersystem if the root string property is satisfied.  The notion of  locally finite root supersystems is a generalization of both locally finite root systems \cite{LN} and generalized root systems \cite{serg}. Indeed, a locally finite root system is nothing but a locally finite root supersystem without  nonsingular roots and generalized root systems are nothing but locally finite root supersystems which are finite.
We classify irreducible  locally finite root supersystems and show that each locally finite root supersystem is a direct union of finite root supersystems with  the same nature.  Locally finite root supersystems appear naturally in the theory of locally finite Lie superalgebras; see  \cite{P} and \cite{You2}.   We hope our classification offers a new approach to the study of these Lie superalgebras.

The paper is arranged as follows. In Section 2, we gather the preliminaries. The third section, which is divided into two subsections, is devoted to locally finite root supersystems. In the first subsection, given a field extension $\bbbk$ of $\bbbf,$ we introduce the notion of  $(\bbbf,\bbbk)$-locally finite root supersystems and study their properties. In the second subsection, we study $(\bbbf,\bbbf)$-locally finite root supersystems which we refer to as locally finite root supersystems.  Using the material of the first subsection, we show that we can define locally finite root supersystems as a subset of a torsion free abelian group and that these  are in correspondence with the ones defined as a subset of an $\bbbf$-vector space; this would be helpful in the study of root graded Lie superalgebras. In the last section, we give the classification of irreducible locally finite root supersystems.
%

\section{Preliminaries}\label{preliminaries}
Throughout this work, $\bbbf$ is a field of characteristic zero. Unless otherwise mentioned, all vector spaces are considered over $\bbbf.$
We  denote the dual space of a  vector space $V$ by  $V^*$ and denote  its group of automorphisms  by $GL(V).$
 For a set $S,$ by $|S|,$ we mean the cardinal number of $S.$ For a map $f:A\longrightarrow B$ and $C\sub A,$ by $f\mid_{_C},$ we mean the restriction of $f$ to $C.$ Also for two symbols $i,j,$ by $\d_{i,j},$ we mean the Kronecker delta.
We finally recall that the direct union is, by definition,  the direct limit of a direct system
whose morphisms are inclusion maps.

\begin{deft}
\label{lfrs}
{\rm Suppose that $\v$ is a nontrivial vector space and $R$ is a subset of $\v.$ We say $R$   is a {\it locally finite root system} in $\v$ (or $(R,\v)$ is a {\it locally finite root system}) if

(i)  $R$ is locally finite, in the sense that  each finite dimensional subspace of $\v$ intersects $R$ in a finite set,

(ii) $0\in R$ and $R$ spans $\v,$

(iii) for $\a\in\rcross:=R\setminus\{0\},$ there is $\check\a\in\v^*$ such that
\begin{itemize}
\item $\check\a(\a) =2,$
\item $\check\a(\b)\in\bbbz;$ $\b\in R,$
\item $R$ is invariant under the reflection  $r_\a$ of $\v$ mapping $v\in\v$ to $v-\check\a(v)\a.$
\end{itemize}
Each element of a locally finite root system $(R,\v)$ is called a {\it root} and  $\dim(\v)$ is called the {\it rank} of $R.$ A finite locally finite root system is called a {\it finite root system}. A bilinear form $\fm$ on $\v$ is called {\it invariant} if it is invariant under the  reflections $r_\a,$ $\a\in \rcross.$ The subgroup $\w$ of $GL(\v)$ generated by $\{r_\a\mid\a\in R^\times\}$ is called the {\it Weyl group} of $R.$  Two locally finite root systems $( R,\u)$
 and $(S,\v)$ are  said  to be {\it isomorphic} if there is a linear isomorphism
$f:\u\longrightarrow \v$ such that $f(R)= S.$}
\end{deft}

Suppose that $(R,\v)$  is a locally finite root   system. A nonempty
subset $S$ of $R$  is said to be  a {\it subsystem} of $R$ if $S$
contains zero and  $r_\a(\b)\in S$ for $\a,\b\in S\setminus\{0\}.$
A subsystem $S$ of $R$ is called {\it full} if $(\hbox{span}_\bbbf
S)\cap R=S.$  A nonempty subset $X$ of $R$ is called {\it
irreducible,} if for each two nonzero elements $\a,\b\in X,$ there exist finitely many nonzero roots
$\a_1:=\a,\a_2,\ldots,\a_n:=\b$ in $X$ such that
$\check\a_{i+1}(\a_{i})\neq 0,$ $1\leq i\leq n-1.$ It is known that the locally finite root system $R$ is a {\it direct sum} of irreducible subsystems in the sense that $R$ is a union of irreducible subsystems $R_i$ ($i\in I$) such that for $i,j\in I$ with $i\neq j$ and $\a\in R_i^\times,$ $\b\in R_j,$   $\check\a(\b)=0$ \cite[\S 3.13]{LN}.  A basis $B$ of the vector space  $\v$ is called a {\it root base} for $R$ if $R\sub(\hbox{span}_{\bbbz^{\geq0}}B)\cup(\hbox{span}_{\bbbz^{\leq0}}B).$ The locally finite root system $R$ has a root base if and only if $R$ is countable \cite[Proposition 6.7 and Theorem 6.9]{LN}.
We take $\{R_\lam\mid
\lam\in\Gamma\}$ to be the class of all finite subsystems of $R,$
and say $\lam\preccurlyeq\mu$ $(\lam,\mu\in\Gamma)$ if $R_\lam$ is
a subsystem of $R_\mu,$ then $(\Gamma,\preccurlyeq)$ is a directed
set and $R$ is the direct union of $\{R_\lam\mid \lam\in\Gamma\}.$
Furthermore, if $R$ is irreducible, it is the direct union of its
irreducible  finite full subsystems; see \cite[Corollarries 3.15 and 3.16]{LN}.
The following Lemma  is well known in the literature \cite[\S 4.14 and Theorem 4.2]{LN}:
\begin{lem}\label{inv-form}
Suppose that $R$ is a locally finite root system in a vector space $\v,$ then $\v$ is equipped with a nondegenerate symmetric invariant  bilinear form $\fm$  such that the form restricted to $\v_\bbbq:=\hbox{span}_\bbbq R$ is a positive definite $\bbbq$-valued bilinear form. In this case for $\a\in \rcross$ and $v\in\v,$ we have $\check\a(v)=2(v,\a)/(\a,\a).$ Moreover, if $R$  is irreducible, then each two  symmetric invariant bilinear forms on $\v$ are proportional; in particular, for each nonzero  symmetric  invariant bilinear form $\fm$ on $\v$ and  $0\neq v\in \v_\bbbq,$ we have $(v,v)\neq 0.$
\end{lem}

Suppose that $T$ is a nonempty  index set  with $|T|\geq 2$ and $\u:=\op_{i\in
T}\bbbf\ep_i$ is the free $\bbbf$-module over   the
set $T.$ Define the  form $$\begin{array}{c}\fm:\u\times\u\longrightarrow\bbbf\\
(\ep_i,\ep_j)\mapsto\d_{i,j}, \hbox{ for } i,j\in T
\end{array}$$
and set
\begin{equation}\label{locally-finite}
\begin{array}{l}
\dot A_T:=\{\ep_i-\ep_j\mid i,j\in T\},\\
D_T:=\dot A_T\cup\{\pm(\ep_i+\ep_j)\mid i,j\in T,\;i\neq j\},\\
B_T:=D_T\cup\{\pm\ep_i\mid i\in T\},\\
C_T:=D_T\cup\{\pm2\ep_i\mid i\in T\},\\
BC_T:=B_T\cup C_T.
\end{array}
\end{equation}
One can see that these are irreducible locally finite root systems
in their $\bbbf-$span's which we refer to as  {\it type} $A,D,B,C$
and $BC$  respectively. Moreover, every irreducible locally finite
root system either is an irreducible finite root system or is isomorphic to one of these root
systems (see \cite[\S4.14,
\S8]{LN}). Now we suppose that $R$ is  an irreducible locally finite
root system; one can define
$$\begin{array}{l}
R_{sh}:=\{\a\in R^\times\mid (\a,\a)\leq(\b,\b);\;\;\hbox{for all $\b\in R$} \},\\
R_{ex}:=R\cap2 R_{sh},\\
R_{lg}:= R^\times\setminus( R_{sh}\cup R_{ex}).
\end{array}$$
The elements of $R_{sh}$ (resp. $R_{lg},R_{ex}$) are called {\it
short roots} (resp. {\it long roots, extra-long roots}) of $R.$ We point  out that following the usual notation in the literature, we use
`` $\dot{}$ '' on the top of $A$  in the list of  locally finite root systems  above as all of them span $\u$ other than the first one which spans a subspace of codimension 1.

In what follows,  we assume that  $R$ is an irreducible locally finite  root system in a vector space $\v$  and $\w$ is its Weyl group.
We fix a  full finite irreducible  subsystem $R_0$ of $R.$ Suppose that   $\Lam$ is an index set and  $\{R_\lam\mid \lam\in\Lam\}$ is the class of all finite irreducible full subsystems of $R$ containing $R_0.$ We  know that $R$ is the direct union of $\{R_\lam\mid \lam\in\Lam\}.$ So  $\v$ is the direct union of $\{\v_\lam\mid \lam\in\Lam\}$ in which  for $\lam\in\Lam,$ $\v_\lam$ is  the linear span of $R_\lam.$ Next, for $\lam\in\Lam,$ set $\w_\lam$ to be the subgroup of $GL(\v_\lam)$ generated by $ r_\a|_{_{\v_\lam}},$ $\a\in R_\lam^\times.$ Since for  $\a\in R_\lam^\times,$ $r_\a(\v_\lam)\sub\v_\lam,$
we can identify the Weyl group of $R_\lam$ with $\w_\lam.$
Now set $$\begin{array}{ll}\fp:=\{v\in\v\mid  \check\a(v)\in\bbbz,\;\;\forall \a\in\rcross\},\\
\fp_\lam:=\{v\in\v_\lam\mid  \check\a(v)\in\bbbz,\;\;\forall \a\in R^\times_\lam\}\;\; (\lam\in\Lam).
\end{array}$$ Each element of $\fp$ is called a {\it weight}. Using Lemma \ref{inv-form}, one can see that  $\w$ (resp. $\w_\lam,$ $\lam\in\Lam$) acts on $\fp$ (resp. $\fp_\lam$) by the natural action.
\begin{deft}\label{small-minimal}\cite[Definition 3.1]{serg} {\rm Consider the action $\w$ on $\fp.$ An orbit $S\sub\fp$ is called {\it small} if for $x,y\in S,$ either $x=\pm y$ or $x-y\in \rcross.$  }\end{deft}

\begin{lem}
  \label{final}
 A subset  $S$ of $\fp$  is a small orbit  if and only if
\begin{itemize}
  \item $\w S\sub S,$
  \item the action of $\w$ on $S$ is transitive,
  \item for each $x\in S,$ there is $\lam_0\in\Lam$ such that  for all $\lam\in\Lam$ with $\lam_0\preccurlyeq\lam,$ $\w_\lam x$ is a  small  orbit in $\fp_\lam.$
\end{itemize}
\end{lem}
\pf Suppose that $S\sub\fp$ is a small orbit and pick a representative $\mathfrak{s}$ of this orbit. Since $\sfs\in\v$ and $\v$ is  the direct union of $\{\v_\lam\mid \lam\in\Lam\},$ one finds $\lam_0\in\Lam$ such that $\sfs\in\v_{\lam_0}.$ So for $\lam\in\Lam$ with $\lam_0\preccurlyeq\lam,$ we have  $\sfs\in\v_\lam$ and $\check\a(\sfs)\in\bbbz$ for all $\a\in\rcross_\lam.$ Therefore $\sfs\in \fp_\lam.$ Now for $x,y\in \w_\lam\sfs\sub\w\sfs,$ we have either  $x=\pm y$ or $x-y\in R.$ In the latter case, $x-y\in \v_\lam\cap R= R_\lam$ as  $\w_\lam\sfs\sub \v_\lam$ and $R_\lam$ is a full subsystem of $R.$   This shows that $\sfs$ is a representative of a small  orbit of the action of $\w_\lam$ on $\fp_\lam.$ Now  it is trivial that if $S$ is a small orbit, then $S$ satisfies the three stated conditions. For the reverse implication, we suppose $S$ is a subset of $\fp$ satisfying the stated conditions, then for $x,y\in S,$ there is $\lam_0\in\Lam$ such that  $\w_\lam x$ is a small orbit of the action  of $\w_\lam$ on $\fp_\lam$ for  all $\lam\in\Lam$ with $\lam_0\preccurlyeq\lam.$ Also we know that there are $\a_1,\ldots,\a_\ell\in \rcross$ such that  $y=r_{\a_1}\cdots r_{\a_\ell}x.$ Now take $\mu\in\Lam$ with $\lam_0\preccurlyeq\mu$  and $\a_1,\ldots,\a_\ell\in R_\mu,$ then $y=wx\in\w_\mu x.$ Therefore, as $\w_\mu x$ is a small orbit for the action of $\w_\mu$ on $\fp_\mu,$ we get either $x=\pm y$ or $x-y\in \rcross.$
\qed
%

\begin{deft}
{\rm Suppose that $R$ is  an irreducible finite root system in a  vector space $\v.$ Suppose that $\fm$ is an invariant bilinear form on $\v.$ If $\D:=\{\a_1,\ldots,\a_\ell\}$ is a root base for $R,$ each element of  the basis $\{\omega_1,\ldots,\omega_\ell\}$ of $\v$ satisfying $(\omega_j,2\a_i)/(\a_i,\a_i)=\d_{i,j},$ $1\leq i,j\leq \ell,$ is  called a {\it fundamental weight} of $R$ (with respect to $\D$).}
\end{deft}

\begin{Pro}\label{smallorbits}
(i) Suppose that  $\ell$ is a positive integer and   $\{\ep_1,\ldots,\ep_{\ell+1}\}$ is a basis for the $\bbbf$-vector space $\v:=\bbbf^{\ell+1}.$ Define $$\fm:\v\times \v\longrightarrow \bbbf;\;\;(\ep_i,\ep_j)\mapsto \d_{i,j};\;1\leq i,j\leq \ell+1.$$ Consider the irreducible finite root system $R:=\{\ep_i-\ep_j\mid 1\leq i,j\leq \ell+1\}$ of type $A_\ell,$ then $\D:=\{\a_1:=\ep_1-\ep_2,\ldots,\a_n:=\ep_\ell-\ep_{\ell+1}\}$ is a root base for $R$ and  $\{\omega_i:=\ep_1+\cdots+\ep_i-\frac{i}{\ell+1}(\ep_1+\cdots+\ep_{\ell+1})\mid 1\leq i\leq \ell\}$ is the set of fundamental weights with respect to this root base. Also if $\ell=1,$ $\{\pm \frac{k}{2}\a_1\}$ $(k\in\bbbz)$ are the only small orbits and if  $\ell\geq2,$  $\w\omega_1$ and $\w\omega_\ell=-\w\omega_1$ are the only small orbits.

(ii) Suppose that $R$ is an irreducible finite root system of type $X\neq A$ and  rank $\ell$ with  Weyl group $\w.$ Suppose that $\D$ is a root base for $R$ and $\Omega$ is the set of fundament weights with respect to $\D.$
Then small orbits for the action of $\w$ on the set of  weights exist if and only if $R$ is of type $X\neq E_{6,7,8}, F_4.$
Moreover, if $R$ is of type $G_2,$ $R_{sh}$ is the only small orbit and if $R$ is of type $X\neq E_{6,7,8}, F_4,G_2,$ up to $\w$-conjugation, we have
$${\tiny
\begin{tabular}{|c|c|c|c|}
\hline
\hbox{ X }& \hbox{rank}&\hbox{   $R$ }& \hbox{$\D$}\\
\hline
\hbox{\tiny$B$}&$\ell\geq2$&$\{0,\pm\ep_i,\pm(\ep_i\pm \ep_j)\mid  1\leq i<j\leq \ell\}$&$\{\a_i:=\ep_i-\ep_{i+1},\a_\ell:=\ep_\ell\mid 1\leq i\leq \ell-1\}$\\
\hline
\hbox{\tiny$C$}&$\ell\geq3$&$\{0,\pm2\ep_i,\pm(\ep_i\pm \ep_j)\mid  1\leq i<j\leq \ell\}$ &$\{\a_i:=\ep_i-\ep_{i+1},\a_\ell:=2\ep_\ell\mid 1\leq i\leq \ell-1\}$\\
\hline
\hbox{\tiny$D$} &$\ell\geq4$&$\{0,\pm(\ep_i\pm \ep_j)\mid  1\leq i<j\leq \ell\}$ &$\{\a_i:=\ep_i-\ep_{i+1},\ep_{\ell}:=\ep_{\ell-1}+\ep_{\ell}\mid 1\leq i\leq \ell-1\}$\\
\hline
\hbox{\tiny$BC$}&$\ell\geq2$&$\{\pm\ep_i,\pm(\ep_i\pm \ep_j)\mid 1\leq i, j\leq  \ell\}$ &$\{\a_i:=\ep_i-\ep_{i+1},\a_{\ell}:=\ep_{\ell}\mid 1\leq i\leq \ell-1\}$\\
\hline
\hbox{\tiny$BC$}&$\ell=1$&$\{\pm\ep_1,\pm2\ep_1\}$ &$\{\a_1:=\ep_1\}$\\
\hline
\end{tabular}
}$$
$${\tiny
\begin{tabular}{|c|c|l|c|}
\hline
\hbox{ X }& \hbox{rank}&  \hbox{$\Omega$}&\hbox{small orbits}\\
\hline
$B$& $\ell=2$&$\{\omega_1=\ep_1,\omega_2=\frac{1}{2}(\ep_1+\ep_2)\}$&$\w\omega_1$\\
\hline
$B$&$\ell= 3$&$\{\omega_i=\ep_1+\cdots+\ep_i,\omega_3=\frac{1}{2}(\ep_1+\cdots+\ep_\ell)\mid 1\leq i\leq 2\}$&$\w\omega_1,\w\omega_3$\\
\hline
$B$& $\ell>3$&$\{\omega_i=\ep_1+\cdots+\ep_i,\omega_\ell=\frac{1}{2}(\ep_1+\cdots+\ep_\ell)\mid 1\leq i\leq \ell-1\}$&$\w\omega_1$\\
\hline
$C$&$\ell\geq3$&$\{\omega_i=\ep_1+\cdots+\ep_i\mid 1\leq i\leq \ell\}$&$\w\omega_1$\\
\hline
$D$ &$\ell=4$&$\{\omega_i=\ep_1+\cdots+\ep_i,\omega_{\ell-1}=\frac{1}{2}(\ep_1+\cdots+\ep_{\ell-1}-\ep_\ell),$&\\
&&$\omega_\ell=\frac{1}{2}(\ep_1+\cdots+\ep_\ell)\mid 1\leq i\leq \ell-2\}$&$\w\omega_1,\w\omega_3,\w\omega_4$\\
\hline
$D$ &$\ell>4$&$\{\omega_i=\ep_1+\cdots+\ep_i,\omega_{\ell-1}=\frac{1}{2}(\ep_1+\cdots+\ep_{\ell-1}-\ep_\ell),$&
\\&&$\omega_\ell=\frac{1}{2}(\ep_1+\cdots+\ep_\ell)\mid 1\leq i\leq \ell-2\}$&$\w\omega_1$\\
\hline
$BC$&$\ell=1$&$\{\omega_1=\frac{1}{2}\ep_1\}$&$\{\pm k\ep_1\}\;(k\in\bbbz)$\\
\hline
$BC$&$\ell=2$&$\{\omega_1=\ep_1,\omega_\ell=\frac{1}{2}(\ep_1+\ep_2)\}$&$\w\omega_1,2\w\omega_2$\\
\hline
$BC$&$\ell\geq3$&$\{\omega_i=\ep_1+\cdots+\ep_i,\omega_\ell=\frac{1}{2}(\ep_1+\cdots+\ep_\ell)\mid 1\leq i\leq \ell-1\}$&$\w\omega_1$\\
\hline
\end{tabular}
}$$

%
%
(iii) Suppose that $T$ is an infinite index set and assume $R$ is a locally finite root system  as in (\ref{locally-finite}). Then there is no small orbit for $R$  if $R$ is of type $\dot A_T$ and  $\w\ep_1$  is  the only small orbit for $R$ if $R$ is of type $B_T,C_T,BC_T$ and $D_T.$
\end{Pro}
\pf
$(i)$ The first statement is immediate. For the last one, one can get the result using an easy verification if $R$ is of rank 1.
Also if the rank of $R$ is greater than 1, we have the result by \cite[Example 3.2]{serg}.
%
%
%
%
%

$(ii)$ If $\ell=1,$ the result follows using a straightforward verification. If $\ell\geq2,$ we have the statement using  \cite[Theorem 3.4]{serg}.

$(iii)$
\noindent\underline{Locally finite root systems of type $\dot A_T$}:
We have $R=\{\ep_i-\ep_j\mid i,j\in T\}.$ We fix two distinct elements of $T$ and call them $1$ and $2.$
We take $\Lam$ to be an index set and $\{T_\lam\mid \lam\in\Lam\}$ to be the class of all  finite subsets of $T$ containing $1,2.$
Then $R_\lam:=\{\ep_i-\ep_j\mid i,j\in T_\lam\}$ is a finite full irreducible subsystem of $R$ and $R=\cup_{\lam\in\Lam}R_\lam.$ Consider the action of Weyl group $\w$ of $R$ on the set of weights $\fp$ of $R$ and let $S$ be a small  orbit of this action. Then  for a fix $\sfs\in S,$ by Lemma \ref{final}, there is $\lam\in \Lam$  such that $\w_\lam\sfs$ is a small  orbit of the action  of $\w_\lam$ on $\fp_\lam.$ By part $(i),$ either  $\w_\lam\sfs=\w_\lam(\ep_1-\frac{1}{\ell_\lam}\sum_{i\in T_\lam}\ep_i)$ or $\w_\lam\sfs=-\w_\lam(\ep_1-\frac{1}{\ell_\lam}\sum_{i\in T_\lam}\ep_i),$ where $\ell_\lam=|T_\lam|;$ in particular, either $\ep_1-\frac{1}{\ell_\lam}\sum_{i\in T_\lam}\ep_i\in S$ or $-\ep_1+\frac{1}{\ell_\lam}\sum_{i\in T_\lam}\ep_i\in S.$ But if $r\in T_\lam$ and $s\in T\setminus T_\lam,$ one sees that $2(\ep_1-\frac{1}{\ell_\lam}\sum_{i\in T_\lam}\ep_i,\ep_r-\ep_s)/(\ep_r-\ep_s,\ep_r-\ep_s)=\d_{1,r}-1/{\ell_\lam}$ which is not an integer number, so $\pm (\ep_1-\frac{1}{\ell_\lam}\sum_{i\in T_\lam}\ep_i)\not\in \fp,$ a contradiction. These all together imply that if $R$ is of type $\dot A_T$ for some infinite index set $T,$ there is no small orbit for  the action of $\w$ on $\fp.$

\medskip

\noindent \underline{Locally  finite root systems of type $B_T,C_T,D_T,BC_T:$}  We just  consider type $B_T.$ The result for the  other types  similarly follows.
Fix three distinct elements of $T$ and call them 1,2 and 3. If $\Lam$ is an index set and $\{T_\lam\mid \lam\in\Lam\}$ is the class of all finite subsets of $T$ containing $1,2,3.$ Then the root system $R$ is the direct union of $R_\lam$'s where for $\lam\in\Lam,$ $R_\lam:=R\cap\hbox{span}_\bbbf\{\ep_i\mid i\in T_\lam\}.$ Now if $\sfs$ is a representative of a small orbit of the action of  Weyl group $\w$ on $\fp,$  there is $\lam\in\Lam$ with $|T_\lam|>4$ such that $\w_\lam\sfs$ is a small orbit for the action of $\w_\lam$ on the set of weights of $R_\lam,$ so by part $(ii),$ $\w_\lam\sfs=\w_\lam\ep_1$ and so $\w\sfs=\w \ep_1.$
\qed

\section{locally finite root supersystems}
Suppose that $\v$ is a vector space over $\bbbf$ and $\bbbk$ is a field extension of $\bbbf.$ By a {\it bilinear form on $\v$ with values in $\bbbk,$} we mean an $\bbbf$-bilinear map $f:\v\times \v\longrightarrow \bbbk.$  The bilinear form $f$ is called {\it symmetric} if $f(u,v)=f(v,u)$ for all $u,v\in \v.$ If $f$ is a symmetric bilinear form with values in $\bbbk,$ the set $\v^0:=\{v\in\v\mid f(v,w)=0;\;\;\forall w\in \v\}$ is called the {\it radical} of $f.$ The symmetric bilinear form $f$ is called {\it nondegenerate} if $\v^0=\{0\}.$ Using a modified version of \cite[Lemma 3.6]{MY}, we get the following lemma.
\begin{lem}\label{yoshii}
Suppose that $\v$ is an $\bbbf$-vector space and $\bbbk$ is a field extension of $\bbbf.$ Suppose that  $\fm:\v\times\v\longrightarrow \bbbk$ is a nondegenerate bilinear  form with values in $\bbbk.$ Then  for a finite dimensional subspace $W$ of $\v,$ there is  a finite dimensional subspace $\u$ containing $W$ on which the restriction of the form is nondegenerate.
\end{lem}
\pf Suppose that $W$
is a finite dimensional subspace of $\v.$ We use induction on dimension of the radical $W^0$ of $W$ to get the result. If the form  is nondegenerate on $W,$ there is nothing to prove. Suppose $\{u_1,\ldots,u_m\}$ is a basis for $W^0$ and extend it to a basis $\{u_1,\ldots,u_m,w_1,\ldots,w_n\}$ for $W.$ Since the form on $\v$ is nondegenerate, there is $x_1\in\v$ such that $(u_1,x_1)\neq 0.$ Now consider the subspace $W_1:=\hbox{span}_\bbbf\{w_1,\ldots,w_n,u_1,\ldots,u_m,x_1\}.$
If $x:=\sum_{i=1}^nr_iw_i+\sum_{i=1}^ms_iu_i+rx_1$ belongs to $W_1^0,$ the radical of the form on $W_1,$ then $(x,u_1)=0.$ This implies that  $r=0.$ Therefore  $x=\sum_{i=1}^nr_iw_i+\sum_{i=1}^ms_iu_i$ and so $x\in W^0.$  This means that  $W_1^0\sub  W^0.$ Now if $\dim(W_1^0)=\dim(W^0),$ then $W^0=W_1^0.$ But $(u_1,x_1)\neq 0$ which shows that $u_1\in W^0\setminus  W_1^0.$ So $\dim(W_1^0)\lneq \dim(W^0).$  Now by induction hypothesis,  there is a finite dimensional subspace $\u$ of $V$ containing $W_1$ (and so containing $W$) on which the form is nondegenerate.\qed

%
%
\subsection{$(\bbbf,\bbbk)$-Locally finite root supersystems}
Throughout this subsection, we assume $\bbbk$ is a field extension of $\bbbf.$ We define two classes $\T_{\bbbf,\bbbk}$ and $\T'_{\bbbf,\bbbk}$ of triples  $(\v,\fm,R)$ consisting of a vector space   $\v$ over $\bbbf,$   a nondegenerate symmetric  bilinear form  $\fm:\v\times\v\longrightarrow\bbbk$ with values in $\bbbk$ and a subset  $R$ of $\v$ satisfying certain axioms. The elements of  $\T_{\bbbf,\bbbk}$ satisfy a so-called {\it local finiteness} property while the elements of   $\T'_{\bbbf,\bbbk}$  satisfy a so-called {\it root string} property. We first investigate the properties of the elements of each class and finally  prove that these two classes coincide.

  For a subset $R$ of an $\bbbf$-vector space $\v$  equipped with a nondegenerate symmetric bilinear form $\fm$ with values in $\bbbk,$ we set $$\begin{array}{l}
R^\times_{re}:=\{\a\in R\mid (\a,\a)\neq0\},\\\\
R_{ns}:=\{\a\in R\mid (\a,\a)=0 \},\\\\
R^\times:=R\setminus\{0\},\;\; \rim^\times:=\rim\setminus\{0\},\;\;\rre:=\rre^\times\cup\{0\}.
\end{array}$$

Set $\T_{\bbbf,\bbbk}$ to be the class all triples $(\v,\fm,R),$ where $\v$ is a vector space over $\bbbf,$ $\fm$ is a nondegenerate symmetric  bilinear form on $\v$ with values in $\bbbk$ and  $R$ is  a subset of $\v$ such that
 \begin{enumerate}
\item $0\in R$ and $\rre$ is locally finite in $\v_{re}:=\hbox{span}_\bbbf\rre,$
\item $R=-R,$
\item $\hbox{span}_\bbbf R=\v,$
\item for $\a\in R^\times_{re}$ and $\b\in R,$ $2(\b,\a)/(\a,\a)\in\bbbz$ and $\b-\frac{2(\b,\a)}{(\a,\a)}\a\in R,$
\item for $\a\in R_{ns}$ and $\b\in R$ with $(\a,\b)\neq 0,$ $\{\b-\a,\b+\a\}\cap R\neq \emptyset.$
\end{enumerate}

\medskip

Next  set $\T_{\bbbf,\bbbk}'$ to be the class of all triples $(\v,\fm,R),$ where $\v$ is a vector space over $\bbbf,$ $\fm$ is a nondegenerate symmetric  bilinear form on $\v$ with values in $\bbbk$ and  $R$ is  a subset of $\v$ such that  \begin{enumerate}
\item $0\in R,$
\item $R=-R,$
\item $\hbox{span}_\bbbf R=\v,$
\item for $\a\in R^\times_{re}$ and $\b\in R,$ $2(\b,\a)/(\a,\a)\in\bbbz$ and $\b-\frac{2(\b,\a)}{(\a,\a)}\a\in R,$
\item for $\a\in\rre^\times$ and $\b\in \rre,$ there are nonnegative integers $p,q$ such that $\{i\in \bbbz\mid \b+i\a\in \rre\}=\{-p,\ldots,q\}$ and $p-q=2(\b,\a)/(\a,\a),$ (we refer to this property as the {\it root string property}),
\item for $\a\in R_{ns}$ and $\b\in R$ with $(\a,\b)\neq 0,$ $\{\b-\a,\b+\a\}\cap R\neq \emptyset.$
\end{enumerate}
\begin{con}{\rm
Throughout this subsection, we always assume for a triple $(\v,\fm,R)\in\T_{\bbbf,\bbbk}\cup\T'_{\bbbf,\bbbk},$ $\rre\neq\{0\}.$}
\end{con}

\begin{deft}{\rm
For a  triple $(\v,\fm,R)$ of $ \T_{\bbbf,\bbbk}\cup T'_{\bbbf,\bbbk},$ we say $(\v,\fm,R),$ or simply $R,$  is  {\it irreducible} if $\rcross$ cannot be written as a disjoint union of two nonempty orthogonal subsets. For $X\sub R,$ we say $\a,\b\in X\setminus\{0\}$ are  {\it connected} in $X$ if there is a chain $\a_1,\ldots,\a_t\in X$ with $\a_1=\a,$ $\a_t=\b$ and $(\a_i,\a_{i+1})\neq0$ for $1\leq i\leq t-1.$ If $\a,\b\in X\setminus\{0\}$ are not connected, we say they are {\it disconnected}. A subset $X$ of $R$ is called {\it connected} if each two nonzero elements of $X$ are connected in $X.$}
\end{deft}

 Suppose that $(\v,\fm,R)\in\T_{\bbbf,\bbbk}\cup \T'_{\bbbf,\bbbk}$ and $\a\in \rre^\times.$ We note that for $v\in \v=\hbox{span}_\bbbf R,$ there are $r_1,\ldots,r_n\in\bbbf$ and $\a_1,\ldots,\a_n\in R$ with $v=\sum_{i=1}^nr_i\a_i,$  so we have \begin{equation}\label{compatible}
2\frac{(v,\a)}{(\a,\a)}=2\frac{\sum_{i=1}^n(r_i\a_i,\a)}{(\a,\a)}=\sum_{i=1}^nr_i\frac{2(\a_i,\a)}{(\a,\a)}\in\bbbf.\end{equation} This allows us to define $$r_\a:\v\longrightarrow \v;\;\; v\mapsto v-\frac{2(v,\a)}{(\a,\a)}\a\;\;\;(v\in\v).$$ The subgroup $\w$ of $GL(\v)$ generated by $\{r_\a\mid \a\in \rcross_{re}\}$ is called the {\it Weyl group} of $R.$

\begin{lem}\label{real roots}
Suppose that $(\v,\fm,R)\in\T_{\bbbf,\bbbk},$ then

(a) for $\a,\b\in R$ and $w\in\w,$ $(w(\a),w(\b))=(\a,\b),$

(b) $R_{re}$ is a locally finite root system in $\v_{re}=\hbox{span}_\bbbf \rre;$ in particular,  if $\rim=\{0\},$ $R$ is a locally finite root system,

(c) for $\a,\b\in R_{re}^\times,$  we have $2(\b,\a)/(\a,\a)\in\{-4,-3,\ldots,3,4\},$

(d) the form $\fm$ restricted to $\v_{re}$ is nondegenerate.
\end{lem}
\begin{proof}
$(a)$ It is easy to see.

$(b)$  For $\a\in\rre^\times,$ we consider (\ref{compatible}) and set $\check\a:=\frac{2(\a,\cdot)}{(\a,\a)}\in\v_{re}^*.$ We note that  for $\a,\b\in R_{re}^\times,$ $r_\a(\b)=\b-\check\a(\b)\a.$ We have $\check\a(\a)=2,$ $\check\a(\b)\in\bbbz$ and that by part $(a),$ $r_\a(R_{re})\sub R_{re}.$ Now as $\rre$ is locally finite in $\v_{re},$ we get that  $R_{re}$ is a locally finite root system in $\v_{re}$.

$(c)$ It follows from part $(b);$ see \cite[\S 3]{LN}.

$(d)$ We know from part ($b$)  that $\rre$ is a locally finite root system in $\v_{re}$ with $\check\a=2(\a,\cdot)/(\a,\a),$ $\a\in\rre^\times.$ So  $1\ot \rre:=\{1\ot \a\mid \a\in \rre\}\sub \bbbk\ot_\bbbf\v_{re}$ is a locally finite root system in $\bbbk\ot_\bbbf\v_{re}$ \cite[\S 4.14]{LN}. Extend the form on $\v_{re}$ to the form $\fm_\bbbk:(\bbbk\ot_\bbbf\v_{re})\times (\bbbk\ot_\bbbf\v_{re})\longrightarrow \bbbk $ defined  by $$(r\ot u,s\ot v)_\bbbk:=rs(u,v);\;\;r,s\in\bbbk,\; u,v\in\v_{re}.$$ We identify $1\ot \rre$ with $\rre$ and note that   $\rre$  is a direct sum of irreducible subsystems $R_i$ ($i\in I$) such that  $\check\a(\b)=0$ for $\a\in R_i\setminus\{0\},$ $\b\in R_j,$ $i,j\in I$ with $i\neq j;$ in particular, we have  $(R_i,R_j)_\bbbk=(R_i,R_j)=\{0\}$ for $i,j\in I$ with $i\neq j.$ On the other hand, by Lemma \ref{inv-form}, $\bbbk\ot_\bbbf\v_{re}$ is equipped with a nondegenerate symmetric invariant bilinear form $\fm':(\bbbk\ot_\bbbf\v_{re})\times (\bbbk\ot_\bbbf\v_{re})\longrightarrow \bbbk$  such that $(R_i,R_j)'=\{0\}$ for $i,j\in I$ with $i\neq j.$ By the same Lemma, as    $\fm_\bbbk|_{_{\hbox{span}_\bbbk R_i}}$ is a nonzero symmetric invariant bilinear form, it is a nonzero scalar multiple of $\fm'|_{_{\hbox{span}_\bbbk R_i}}.$ Now it follows  that $\fm_\bbbk$ on $(\bbbk\ot_\bbbf\v_{re})\times(\bbbk\ot_\bbbf\v_{re})$ is nondegenerate. This in turn implies that the form restricted to $\v_{re}$ is nondegenerate.
\end{proof}


\begin{lem}
\label{rational}  Suppose that $(\v,\fm, R)\in\T_{\bbbf,\bbbk}\cup\T'_{\bbbf,\bbbk},$ then we have the following:

(i) If $\a,\b\in\rre^\times$ are connected in $\rre,$ then $(\a,\a)/(\b,\b)\in\bbbq.$

(ii) Each subset of $\rre^\times$ whose elements are mutually disconnected in $\rre^\times$ is linearly independent.
\end{lem}
\pf ($i$) Since  $\a,\b$ are connected in $\rre,$ there are  $\a_1,\ldots,\a_m\in\rre^\times$ with $\a_1=\a,$ $\a_m=\b$ and $(\a_i,\a_{i+1})\neq 0$ for all $1\leq i\leq m-1.$ For $1\leq i\leq m-1,$ $\frac{2(\a_i,\a_{i+1})}{(\a_i,\a_i)},\frac{2(\a_{i+1},\a_i)}{(\a_{i+1},\a_{i+1})}\in\bbbz\setminus\{0\}.$ Therefore $\frac{(\a_i,\a_i)}{(\a_{i+1},\a_{i+1})}\in\bbbq$ for all $1\leq i\leq m-1$ and consequently,  we have $(\a,\a)/(\b,\b)\in\bbbq.$

($ii$) Suppose that $\{\a_j\mid j\in J\}$ is a subset of  $\rre^\times$ whose elements are mutually disconnected in $\rre^\times.$ If there is  $j_0\in J$ such that $\a_{j_0}=\sum_{j_0\neq j\in J}r_j\a_j$ for some $r_j\in \bbbf$ ($j_0\neq j\in J$), then $0\neq(\a_{j_0},\a_{j_0})=(\a_{j_0},\sum_{j_0\neq j\in J}r_j\a_j)=0,$ a contradiction. This completes the proof.
\qed

\smallskip

Using the same argument as in \cite[Lemmas 1.8, 1.10]{serg}, we have the following lemma.
\begin{lem}
\label{serg2} Suppose that  $(\v,\fm,R)\in\T_{\bbbf,\bbbk}\cup\T'_{\bbbf,\bbbk},$ then  we have the following:

{\rm (1)} If $\a\in R_{re}^\times$ and $\b\in R_{ns},$ then $2(\a,\b)/(\a,\a)\in\{0,\pm1,\pm2\}.$

{\rm (2)}  If $\a,\b\in\rim$ with $(\a,\b)\neq0$ and $k\in\bbbz,$ then $\b+k\a\in R$ only if $k=0,\pm1;$ in particular, for  $\a,\b\in\rim$ with $(\a,\b)\neq0,$  $|\{\b+k\a\mid k\in \bbbz\}|\leq3.$

\end{lem}%

\begin{lem} \label{bounded} Suppose that $(\v,\fm, R)\in\T_{\bbbf,\bbbk}',$ then
 $\{2(\b,\a)/(\a,\a)\mid \a\in \rre^\times,\b\in R\}$ is bounded
\end{lem}
\pf  We first note that by Lemma \ref{serg2}, for $\a\in\rre^\times$ and $\b\in \rim,$ $\frac{2(\b,\a)}{(\a,\a)}\in\{0,\pm1,\pm2\}.$ So  we assume  $\a,\b\in \rre^\times$ and show that  $-9\leq 2(\b,\a)/(\a,\a)\leq 9.$ To the contrary, suppose that it is not true. Replacing $\b$ with $-\b$ if it is necessary, we assume  $\a,\b\in\rre^\times $ and  $a:=2(\b,\a)/(\a,\a)\leq-10.$ We know that there are nonnegative  integers $p,q$ such that
$$\{k\in\bbbz\mid \b+k\a\in \rre\}=\{-p,\ldots,q\}\;\hbox{and}\; p-q=2(\b,\a)/(\a,\a)=a\leq -10.$$
So $\b+2\a,\b+3\a\in \rre.$
Since $(\b,\a)\neq 0,$ we get $b:=2(\b,\a)/(\b,\b)\in\bbbz\setminus\{0\}.$ We consider the following three cases:
\begin{itemize}
\item $b=-1:$ Since $a\leq -10,$ we get $1+(9/a)>0.$ Now we have
\begin{eqnarray*}\frac{2(\b,\b+3\a)}{(\b+3\a,\b+3\a)}=\frac{2(\b,\b)+6(\b,\a)}{-(\b,\b)(2+9/a)}
=\frac{(\b,\b)(2+3b)}{-(\b,\b)(2+9/a)}
&=&\frac{1}{2+9/a}\not\in\bbbz,
\end{eqnarray*}a contradiction.

\item  $b\neq -1$ and $a\neq 4b:$ We have \begin{equation}
\label{b}\frac{b}{b+1}>0,\;\; (1+\frac{4}{a})>0\andd \frac{b}{b+1}(1+\frac{4}{a})\neq 1.
\end{equation}
So
\begin{eqnarray*}
\frac{2(\b,\b+2\a)}{(\b+2\a,\b+2\a)}&=&\frac{2(\b,\b)+4(\b,\a)}{(\b,\b)+4(\b,\a)+4(\a,\a)}\\
&=&\frac{2(\b,\b)(1+2(\b,\a)/(\b,\b))}{(\b,\b)(1+4(\b,\a)/(\b,\b)+4(\a,\a)/(\b,\b))}\\
&=&2\frac{1+b}{1+2b+4b/a}\\
&=&2\frac{1}{1+\frac{b}{b+1}(1+\frac{4}{a})}.
\end{eqnarray*}
But by (\ref{b}), $1\neq 2\frac{1}{1+\frac{b}{b+1}(1+\frac{4}{a})}\in\bbbq$  and $0<2\frac{1}{1+\frac{b}{b+1}(1+\frac{4}{a})}<2.$ So  $2(\b,\b+2\a)/(\b+2\a,\b+2\a)\not\in\bbbz$ which is a contradiction.
\item $b\neq -1$ and $a=4b:$
We have
\begin{eqnarray*}
\frac{2(\a,\b+2\a)}{(\b+2\a,\b+2\a)}&=&\frac{2(\a,\b)+4(\a,\a)}{2(\b,\b)(b+1)}\\
&=&\frac{(\b,\b)(2(\a,\b)/(\b,\b)+4(\a,\a)/(\b,\b))}{2(\b,\b)(b+1)}\\
&=&\frac{(\b,\b)(b+4b/a)}{2(\b,\b)(b+1)}\\
&=&\frac{b+1}{2(b+1)}=1/2\not\in\bbbz,
\end{eqnarray*}
a contradiction.
\end{itemize} These all together complete the proof.
\qed

\begin{lem}
The class $\T_{\bbbf,\bbbk}$ coincides with the class $\T_{\bbbf,\bbbk}'.$
\end{lem}
\pf We first suppose that $(\v,\fm,R)\in \T_{\bbbf,\bbbk}.$ To show that $(\v,\fm,R)\in \T'_{\bbbf,\bbbk},$ we need to prove that the root string property holds for $(\v,\fm,R).$ But it is immediate as  $\rre$ is a locally finite root system in $\v_{re}$ by Lemma \ref{real roots}; see \cite[\S 3]{LN}.

Conversely, suppose that $(\v,\fm,R)\in\T_{\bbbf,\bbbk}'.$ We show that  $(\v,\fm,R)\in\T_{\bbbf,\bbbk}.$
 For this, we need to prove that $R_{re}$ is  locally finite in $\v_{re}.$ Suppose that $W$ is a finite dimensional subspace of $\v_{re}.$ We prove that $W\cap\rre$ is  a finite set. Since the form is nondegenerate on $\v,$ by Lemma \ref{yoshii}, there is  a finite dimensional subspace $U$ of $\v$ such that the form restricted to $U$ is nondegenerate and $ W\sub U.$ Since $U$ is finite dimensional,  there is a finite subset $\{\a_1,\ldots,\a_t\}\sub R$ with $U\sub\hbox{span}_\bbbf\{\a_1,\ldots,\a_t\}.$ To complete the proof, it is enough to show $U\cap\rre$ is finite.  We note that connectedness is an equivalence relation on $\rre^\times$ and so $\rre^\times$ is decomposed into connected   components $S_j,$ where $j$ runs over a nonempty index set $J.$ Using Lemma \ref{rational}($ii$) and considering the fact that $U$ is finite dimensional, it is enough to   show that $U\cap S_j$ is a finite set for all $j\in J.$ Suppose that $j\in J$ and consider the map $$\begin{array}{l}
\varphi: S_j\cap U\longrightarrow \bbbz^t\\
\a\mapsto(2(\a,\a_1)/(\a,\a),\ldots,2(\a,\a_t)/(\a,\a)).
\end{array}$$
Since by Lemma \ref{bounded}, $\{2(\b,\gamma)/(\gamma,\gamma)\mid \b\in \rre^\times,\gamma\in R\} $ is bounded, we get that $im \varphi$ is a finite set. On the other hand, as the form on $U$ is nondegenerate, $\varphi$ is one to one; indeed, suppose  $\a,\b\in S_j\cap U$ and $2(\a,\a_i)/(\a,\a)=2(\b,\a_i)/(\b,\b)$ for all $1\leq i\leq t.$ By Lemma \ref{rational}($i$), $(\a,\a)/(\b,\b)\in\bbbq,$ so  we get  for all $1\leq i\leq t,$ $(\a-\frac{(\a,\a)}{(\b,\b)}\b,\a_i)=0.$ Therefore  $(\a-\frac{(\a,\a)}{(\b,\b)}\b,U)=\{0\}.$
So $\a=\frac{(\a,\a)}{(\b,\b)}\b$ as the form on $U$ is nondegenerate.
But $\frac{2(\a,\b)}{(\a,\a)},\frac{2(\a,\b)}{(\b,\b)}\in \bbbz$ which implies  that $(\a,\a)/(\b,\b)\in\{\pm1,\pm 2,\pm\frac{1}{2}\}.$ If $(\a,\a)/(\b,\b)=\pm2,$ then $\a=\pm2\b$ and so $(\a,\a)/(\b,\b)=4,$ a contradiction. Also if $(\a,\a)/(\b,\b)=\pm(1/2),$ then $\a=\pm(1/2)\b$ and $(\a,\a)/(\b,\b)=1/4$ that is again  a contradiction. Finally if $(\a,\a)=-(\b,\b),$ then $\a=-\b$ and so $-1=(\a,\a)/(\b,\b)=1$ that is absurd. Therefore,  $\a=\b.$ Now  $\varphi$ is one to one and $im\varphi$ is a finite set, so $S_j\cap U$ is a finite set.  This completes the proof.
\qed

\begin{deft}
{\rm We call a triple $(\v,\fm, R)\in \T_{\bbbf,\bbbk}=\T'_{\bbbf,\bbbk},$ an {\it $(\bbbf,\bbbk)$-locally finite root supersystem}}.
\end{deft}

\begin{lem}
\label{real-im}
Suppose that $(\v,\fm, R)$ is an $(\bbbf,\bbbk)$-locally finite root supersystem, then for $\a\in \rre^\times$ and $\b\in R,$ there are nonnegative integers $p,q$ such that $\{i\in \bbbz\mid \b+i\a\in R\}=\{-p,\ldots,q\}$ and $p-q=2(\b,\a)/(\a,\a).$
\end{lem}
\pf By Lemma \ref{real roots}, $\rre$ is a locally finite root system and so $\rre^\times=\cup_{i\in I} R_i,$ in which each $R_i$ is a connected component of $\rre^\times.$ Let $\a\in \rre^\times$ and $\b\in R.$ Since the only scalar multiples of $\a$ which can be  roots are $0,\pm\a,\pm(1/2)\a,\pm2\a,$ we are done if $\b=0.$ We next  suppose that $\a,\b\in \rre^\times,$ then there are $i,j\in I$ with $\a\in R_i$ and $\b\in R_j.$ Suppose that  $i=j$ and that $k$ is an integer such  that $\b+k\a\in R.$ Since $\b+k\a\in\hbox{span}_{\bbbq}R_i,$ by the proof of Lemma \ref{inv-form} either $\b+k\a=0,$ or $(\b+k\a,\b+k\a)\neq 0.$ In both cases $\b+k\a\in\rre.$ This implies that $\{k\in\bbbz\mid \b+k\a\in R\}=\{k\in\bbbz\mid \b+k\a\in \rre\}$ and so we are done.

Next suppose $i\neq j.$ Assume  $k\in\bbbz\setminus\{0\}$ and $\b+k\a\in R.$ Since $(\b+k\a,\a)\neq 0$ and $(\b+k\a,\b)\neq 0,$  we have $\b+k\a\not\in \rre.$ So $0=(\b+k\a,\b+k\a)=(\b,\b)+k^2(\a,\a).$ This  in turn implies that $(\b,\b)/(\a,\a)=-k^2.$ Suppose that $|k|>1,$ then since $(\b+k\a,\a)\neq 0,$ there is $r\in\{\pm 1\}$ with $\b+(k+r)\a\in R.$ As above, we get that  $(k+r)^2=-(\b,\b)/(\a,\a)=k^2$ that is a  contradiction. So $|k|=1.$ Now as $r_\a(\b+k\a)=\b-k\a,$ $\{ \b+k\a\mid k\in\bbbz\}\cap R$ is either $\{\b\}$ or $\{\b-\a,\b,\b+\a\}.$

Now we assume  $\a\in \rre^\times$ and  $\b\in\rim^\times.$ By Lemma  \ref{serg2}(1),  we have $n:=-2\frac{(\b,\a)}{(\a,\a)}\in\{0,\pm1,\pm2\}.$ Changing the role of $\a$ with $-\a$ if it is necessary,  we may assume that $n\in\{0,1,2\}.$

Case 1. $n=0:$  We prove that $\{\b+k\a\mid k\in\bbbz\}\cap R$ is one of the following sets: $\{\b\},$ $\{\b-\a,\b,\b+\a\}$ or $\{\b-2\a,\b-\a,\b,\b+\a,\b+2\a\}.$ We first note that since $(\a,\b)=0,$ we have $\b-k\a=r_\a(\b+k\a)$ for $k\in\bbbz.$  So $\b+k\a\in R$ if and only if $\b-k\a\in R.$ Suppose that for some  $k\in\bbbz\setminus\{0\},$  $\b+k\a\in R,$ then $(\b+k\a,\b+k\a)=k^2(\a,\a)\neq0,$ i.e., $\b+k\a\in\rre^\times.$  Now as $2(\a,\b+k\a)/(\b+k\a,\b+k\a)=2/k\in\bbbz,$ we must have $k=\pm1,\pm2.$ To complete the proof, we need to show that if $\gamma:=\b+2\a\in R,$ then $\b+\a\in R.$ For this, we suppose that $\gamma\in R,$ then  we have  $\a,\gamma\in \rre,$   $\gamma+\a=\b+3\a\not\in \rre$ and  $(\gamma,\a)\neq 0.$ Therefore  the root string property implies that $\b+\a=\gamma-\a\in \rre.$

Case 2. $n=1,2:$ We claim  that $\{\b+k\a\mid k\in\bbbz\}\cap R=\{\b,\ldots,\b+n\a\}.$ We first show that $\b+m\a\not\in R$ for $m\in\bbbz^{>n}.$ Suppose to the contrary that $m\in\bbbz^{>n}$ and $\b+m\a\in R.$ If $\b+m\a\in \rim,$ then $0=(\b+m\a,\b+m\a)=m^2(\a,\a)+2m(\b,\a).$
This implies that $-n=2(\b,\a)/(\a,\a)=-m,$ a contradiction. So $\gamma:=\b+m\a\in \rre.$ Using the root string property, we find positive integers $p,q$ with $p-q=2(\gamma,\a)/(\a,\a)=-n+2m$ such that $\{k\in\bbbz\mid \gamma+k\a\in \rre\}=\{-p,\ldots,q\}.$ Since $p-q=-n+2m,$ we get that $p\geq -n+2m,$ so $\b+(m-k)\a=\gamma-k\a\in \rre,$ for $0\leq k\leq 2m-n;$ in particular, $\b=\gamma-m\a\in \rre$ that is  a contradiction.
Next we show that for $1\leq k\leq n,$ $\b+k\a\in R.$ We know that $\b+n\a=r_\a(\b),$ so if $n=1,$ there is nothing to prove. If $n=2,$ then since $(\b+2\a,\a)=(\a,\a)\neq0,$ we get that either  $\b+3\a=(\b+2\a)+\a\in R$ or  $\b+\a=(\b+2\a)-\a\in R.$ But as we have already seen, $\b+3\a\not\in R,$ so $\b+\a\in R.$ We finally show that $\b-k\a\not\in R$ for $k\in\bbbz^{>0}.$ Suppose that $k\in\bbbz^{>0}$ and $\b-k\a\in R,$ then if $(\b-k\a,\b-k\a)=0,$ we get $-n=2(\b,\a)/(\a,\a)=k,$ a contradiction. So $\eta:=\b-k\a\in\rre.$  Therefore, by the root string property, there are positive integers $p,q$ such that $p-q=2(\eta,\a)/(\a,\a)=2(\b,\a)/(\a,\a)-2k=-n-2k$ and $\{t\in\bbbz\mid \eta+t\a\in \rre\}=\{-p,\ldots,q\}.$ So for $0\leq t\leq n+2k,$ we have $\eta+t\a\in \rre,$ in particular, $\b=\eta+k\a\in\rre$ which is again  a contradiction.
\qed

\begin{lem}
\label{general}
Suppose that $(\v,\fm,R)$ is an  irreducible $(\bbbf,\bbbk)$-locally finite root supersystem. Then there is no  class $\{\rim^t\mid t\in T\}$  of nonempty subsets of   $\rim\setminus\{0\}, $ where   $T$  is an index set with $|T|>1,$ such that
 \begin{itemize}
 \item for $t\in T,$ $\rim^t$ is invariant under the Weyl group,
 \item $\rim\setminus\{0\}=\uplus_{t\in T}\rim^t,$
 \item for $t,t'\in T$  with $t\neq t',$ $(\rim^t,\rim^{t'})=\{0\}.$
 \end{itemize}
 \end{lem}
  \pf Using the same argument as in \cite[Proposition 1.15]{serg}, one can prove the  lemma, but for the convenience of the readers, we give a sketch of the proof. Suppose that  $T$ is an index set with $|T|>1$ and $\{\rim^t\mid t\in T\}$ is a class  of nonempty subsets of   $\rim\setminus\{0\} $
as in the statement.   For $t\in T,$ we set $$\rre^t:=\{\a\in \rre\mid (\a,\rim\setminus\rim^t)=\{0\}\}.$$ The proof is carried out in the following steps:
\smallskip

Step 1. For $t\in T,$ $\rre^t$ is invariant under the Weyl group: It follows from Lemma \ref{real roots}($a$) together with the fact that $\rim^t$ is invariant under the Weyl group.
\smallskip

Step 2. $\rre=\cup_{t\in T}\rre^t:$ Suppose that it is not true. So there is $\a\in\rre^\times$ such that $\a\not \in \cup_{t\in T}\rre^t.$ Fix $t_0\in T,$ then $\a\not\in \rre^{t_0},$ so there is $t_1\neq t_0$ and $\d\in\rim^{t_1}$ such that $(\a,\d)\neq 0.$ Again $\a\not\in\rre^{t_1},$ so there is $t_2\in  T\setminus\{t_1\}$ and $\gamma\in\rim^{t_2}$ such that  $(\a,\gamma)\neq 0.$ These imply that  $(\d,\gamma),(r_\a\d,\gamma)\in (\rim ^{t_1},\rim^{t_2})=\{0\}.$ Therefore we have $$0=(\d-r_\a\d,\gamma)=\frac{(2(\a,\d)\a,\gamma)}{(\a,\a)}=2\frac{(\a,\d)(\a,\gamma)}{(\a,\a)}$$ and so $(\a,\d)(\a,\gamma)=0$ which is a contradiction.

\smallskip
Step 3. If $t,t'\in T$ and $t\neq t',$ we have $\rre^t\cap\rre^{t'}=\{0\}:$ Suppose that $\rre^t\cap\rre^{t'}\neq\{0\}.$ We know that $\rre^t\cap\rre^{t'}=\{\a\in \rre\mid (\a,\rim)=\{0\}\}.$ Set $R':=\rre^t\cap\rre^{t'}.$ We claim that $(R',R\setminus R')=\{0\}.$ Suppose that $\d\in R\setminus R'$ and $\gamma\in R'.$ We just need to assume $\d\in\rre.$ Since $\d\not\in R',$ one finds $\eta\in \rim$ such that $(\d,\eta)\neq 0.$ Since $\gamma\in  R',$ we have $r_\d\gamma\in R'$ and so $(\eta,\gamma),(\eta,r_\d\gamma)\in(\rim,R')=\{0\}.$ Therefore as before, we have $(\d,\gamma)(\d,\eta)=0.$ Thus $(\d,\gamma)=0.$ So $R^\times=R'\setminus\{0\}\uplus(R^\times\setminus R')$ with $(R'\setminus\{0\},R^\times\setminus R')=\{0\}.$ This  is a contradiction as $R$ is irreducible.
\smallskip

Step 4. If $t,t'\in T$ and $t\neq t',$ we have $(\rre^t,\rre^{t'})=\{0\}:$
Suppose that $\a\in \rre^t\setminus\{0\}$ and $\b\in \rre^{t'}\setminus\{0\}.$ Using Step 3,
one finds $\d\in \rim^t$ such that $(\d,\a)\neq 0.$ Now  we have $(\d,\b),(r_\a\d,\b)\in(\rim^t,\rre^{t'})=\{0\},$ so it follows that  $(\a,\d)(\a,\b)=0$ and so $(\a,\b)=0.$

Now for a fixed element $t_0\in T,$ set  $R_1:=\cup_{t\in T\setminus\{t_0\}}(\rre^{t}\cup\rim^{t})$ and $R_2:=\rre^{t_0}\cup\rim^{t_0}.$  We  have $R^\times=(R_1\setminus\{0\})\uplus(R_2\setminus\{0\})$  and $(R_1,R_2)=\{0\}$ which contradict the irreducibility of  $(\v,\fm,R).$ This completes the proof.\qed

\medskip
\begin{Pro}\label{conjugate} Suppose that $(\v,\fm,R)$ is an $(\bbbf,\bbbk)$-locally finite root supersystem, then

(i) for   $\d,\gamma\in \rim^\times$ with $(\d,\gamma)\neq 0,$ we have  $\gamma\in\w\d\cup-\w\d,$

(ii) if $(\v,\fm,R)$ is irreducible and $\d\in \rim^\times,$ then   $\rim^\times=\w\d\cup-\w\d.$
\end{Pro}
\pf ($i$) Suppose that $\d,\gamma\in \rim^\times$ such that $(\d,\gamma)\neq0.$ Since $(\d,\gamma)\neq 0,$  there is $s\in\{\pm1\}$ such that $\a:=\d+s\gamma\in R_{re}.$ Also  $\d\in\rim,$
so we get that $0=(\d,\d)=(\a-s\gamma,\a-s\gamma)=(\a,\a)-2(\a,s\gamma)$ and so $2(\a,s\gamma)/(\a,\a)=1.$ Therefore, we have $r_\a(s\gamma)=s\gamma-\frac{2(\a,s\gamma)}{(\a,\a)}\a=s\gamma-\a=-\d.$ So we have $\gamma=-sr_\a(\d).$ This implies that $\gamma\in\w\d\cup-\w\d.$

($ii$) For $\d,\d'\in\rim^\times,$ we say $\d\sim\d'$ if  $\d\in\w\d'\cup(-\w\d').$ It is an equivalence relation. Take $\{S_k\mid k\in K\}$ to be the family of all equivalence classes. If $|K|=1,$ we are done, so suppose that $|K|>1.$ For $k\in K,$ pick $\d_k\in\rim^\times$  such that $S_k=\w\d_k\cup-\w\d_k.$ Then  $\rim^\times=\uplus_{k\in K} S_k$ and each $S_k$   is $\w$-invariant. Also using part ($i$), one gets that  $(S_k,S_{k'})=\{0\},$ for $k,k'\in K$ with $k\neq k'.$  This contradicts  Lemma \ref{general} and so we are done.
\qed
\subsection{Locally finite root supersystems}
\begin{deft}
{\rm We call an {\small $(\bbbf,\bbbf)$}-locally finite root supersystem {\small $(\v,\fm,R),$} a {\it locally finite root supersystem}. If there is no confusion, we say $R$ is a locally finite root supersystem in $\v.$}
\end{deft}

\begin{Example}\label{lf}{\rm
Suppose that $\ell$ is a positive integer and $S_1,S_2$ are two finite root systems of type $A_\ell$ in vector spaces $\u_1,\u_2,$ respectively. As in the previous section, we assume $S_1=\{\ep_i-\ep_j\mid 1\leq i,j\leq \ell+1\}$  and set $\dot\ep_i:=\ep_i-\frac{1}{\ell+1}(\ep_1+\cdots+\ep_{\ell+1})$ $(1\leq i\leq \ell).$ We also take $S_2=\{\d_i-\d_j\mid 1\leq i,j\leq \ell+1\}$ and set $\dot\d_i:=\d_i-\frac{1}{\ell+1}(\d_1+\cdots+\d_{\ell+1})$ $(1\leq i\leq \ell).$  Next we put $\v:=\u_1\op\u_2$ and define $$\begin{array}{l}\fm:\v\times\v\longrightarrow \bbbf\\
(\u_1,\u_2)=\{0\},\\ (\ep_i-\ep_j,\ep_{i'}-\ep_{j'})=\d_{i,i'}-\d_{i,j'}-\d_{j,i'}+\d_{j,j'},\\
(\d_i-\d_j,\d_{i'}-\d_{j'})=-(\d_{i,i'}-\d_{i,j'}-\d_{j,i'}+\d_{j,j'})\;\;\;\;(1\leq i,i',j,j'\leq \ell+1).\end{array}$$ Then one can easily check that $$\begin{array}{l}R_1:=S_1\cup S_2\cup\pm\{\dot\ep_i+\dot\d_j\mid 1\leq i,j\leq \ell+1\},\\
 R_2:=S_1\cup S_2\cup\pm\{\dot\ep_i-\dot\d_j\mid 1\leq i,j\leq \ell+1\}\end{array}$$  are locally  finite root supersystems in $\v.$}
\end{Example}
\begin{deft}{\rm
Suppose that $(\v,\fm, R)$ is a locally finite root supersystem.
\begin{itemize}
\item  Each element of  $R$ is called a {\it root}. Elements of $\rre$ (resp. $\rim$) are called {\it real} (resp. {\it nonsingular}) roots.
\item   A subset $S$ of $R$ is called {\it sub-supersystem} if the restriction of the form to $\hbox{span}_\bbbf S$ is nondegenerate, $0\in S,$ for $\a\in S\cap\rre^\times, \b\in S$ and $\gamma\in S\cap\rim$ with $(\b,\gamma)\neq 0,$ $r_\a(\b)\in S$ and  $\{\gamma-\b,\gamma+\b\}\cap S\neq \emptyset.$
\item If $\{R_i\mid i\in Q\}$ is a class of  sub-supersystems of $R$ which are mutually   orthogonal and $R\setminus\{0\}=\uplus_{i\in Q}(R_i\setminus\{0\}),$  we say $R$ is {\it the direct sum} of $R_i$'s and write $R=\op_{i\in I}R_i.$
%
\item The locally finite root supersystem $R$ is called a {\it finite root supersystem}  if $R$ is a finite subset of $\v.$  \item
Two irreducible locally finite root supersystems $(\v,\fm_1,R)$ and $(\w,\fm_2,S)$ are called {\it isomorphic} if there is a linear isomorphism $\varphi:\v\longrightarrow \w$ and a nonzero scalar $r\in\bbbf$ such that $\varphi(R)=S$ and  $(v,w)_1=r(\varphi(v),\varphi(w))_2$ for all $v,w\in \v.$
\end{itemize}} \end{deft}
\begin{rem}\label{rem3}
{\rm (i) The systematic study of finite root systems with possibly nonsingular roots have been initiated by V. Serganava in 1996. She introduced generalized root systems which are finite root supersystems in our sense.

(ii) If two irreducible locally finite root supersystems $(\v,\fm_1,R)$ and $(\w,\fm_2,S)$ are  isomorphic, then $ \rre$ is isomorphic to $S_{re}.$}
\end{rem}
\begin{Example} {\rm Consider the notations as in Example \ref{lf}. Then  $R_1$ and $R_2$ are irreducible locally finite root supersystems. Moreover, $B:=\{\ep_i-\ep_{\ell+1},\d_i-\d_{\ell+1}\mid 1\leq i\leq \ell\}$ is a basis for $\v$ and $\varphi:\v\longrightarrow \v$ mapping $\ep_i-\ep_{\ell+1}\mapsto \ep_i-\ep_{\ell+1}$ and $\d_i-\d_{\ell+1}\mapsto -(\d_i-\d_{\ell+1})$ ($1\leq i\leq \ell$) is a linear isomorphism satisfying $\varphi(R_1)=R_2$ and $(u,v)=(\varphi(u),\varphi(v))$ for all $u,v\in\v;$ in other words, $R_1$ and $R_2$ are isomorphic.}
\end{Example}

\begin{lem}\label{decom}
Suppose that $R$ is a locally finite root supersystem, then  we have the following statements:

(i) Connectedness defines an equivalence relation on $R^\times.$  If  $\{S_i\mid i\in Q\}$ is the class of connected components of $R^\times,$ then for  $i\in Q,$   $R_i:=S_i\cup\{0\}$ is a sub-supersystem of $R$; in particular, $R$ is a direct sum of irreducible  sub-supersystems. Moreover, $R$ is irreducible if and only if  $R$ is connected.

(ii) If $R\setminus\{0\}=R_1\uplus R_2,$ where $R_1$ and $R_2$ are two nonempty orthogonal subsets of $R,$ then $R_1\cup\{0\}$ and $R_2\cup\{0\}$ are sub-supersystems of $R.$
\end{lem}
\pf ($i$)  It is easy to see that connectedness is an equivalence relation in  $R^\times;$ we just note that for $\a\in R^\times,$ since $R$ spans $\v$ and the form $\fm$ is nondegenerate, there is $\b\in R$ with $(\a,\b)\neq 0.$ This means that $\a$ is connected to $\a$ through the chain $\a,\b,\a.$ Setting $\v_i:=\hbox{span}_\bbbf(S_i),$ one can see  that $\v=\op_{i\in Q}\v_i$ and that the form restricted to $\v_i$ is nondegenerate. Now to complete the proof, it is enough to show that for $\b,\a\in  R_i$ ($i\in Q$), we have $\{\b+k\a\mid k\in\bbbz\}\cap R=\{\b+k\a\mid k\in\bbbz\}\cap R_i.$ So suppose $i\in Q,$ $\a,\b\in R_i$ and $k\in\bbbz$ with $0\neq\b+k\a\in R=\cup_{j\in Q}R_j.$ Then
$(\b+k\a,\sum_{i\neq j\in Q }\v_j)\sub(\v_i,\sum_{i\neq j\in Q }\v_j)=\{0\}.$ But the form is nondegenerate on $\v,$ so there is at least a root in $R_i$ which is not orthogonal to $\b+k\a.$ This means that $\b+k\a\in R_i.$ This completes the proof.

($ii$) It follows  using the same argument as in the previous part.\qed
\medskip

In the following Lemma, we prove that we can define a locally finite root supersystem as a subset of a torsion free abelian group instead of a subset of a vector space.

\begin{lem}
 Suppose that $A$ is an additive abelian group and $\fm:A\times A\longrightarrow \bbbf$ is a group bi-homomorphism, that is $(a+b,c)=(a,c)+(b,c)$ and $(a,b+c)=(a,b)+(a,c)$ for all $a,b,c\in A.$ Suppose that $(a,b)=(b,a)$ for all $a,b\in A$ and  that $A^0:=\{a\in A\mid (a,b)=0\;\;\;\forall b\in A\}=\{0\}.$ Suppose that $R\sub A$ and set$$\begin{array}{l}
\rcross:=R\setminus \{0\}\\\\
\rcross_{re}:=\{\a\in R\mid (\a,\a)\neq0\},\;\;\;\rre:=\rcross_{re}\cup\{0\},\\\\
R_{ns}:=\{\a\in R\mid (\a,\a)=0 \}, \;\;\rim^\times:=R_{ns}\setminus\{0\}.
\end{array}$$
Suppose that the following statements hold:
$$\begin{array}{ll}
(S1)& \hbox{$0\in R,$ and $\hbox{span}_\bbbz(R)= A,$}\\\\
(S2)& \hbox{$R=-R,$}\\\\
(S3)&\hbox{for $\a\in \rre^\times$ and $\b\in R,$ $2(\a,\b)/(\a,\a)\in\bbbz$ and $\b-\frac{2(\b,\a)}{(\a,\a)}\a\in R,$}\\\\
(S4)&\parbox{4.5in}{for $\a,\b\in \rcross_{re},$  there are nonnegative  integers  $p,q$  with $2(\b,\a)/(\a,\a)=p-q$ such that \begin{center}$\{\b+k\a\mid k\in\bbbz\}\cap \rre=\{\b-p\a,\ldots,\b+q\a\},$\end{center}
} \\\\
(S5)&\parbox{4.5in}{for $\a\in \rim$ and $\b\in R$ with $(\a,\b)\neq 0,$
$\{\b-\a,\b+\a\}\cap R\neq\emptyset.$}
\end{array}$$
Extend the map $\fm$ to the  $\bbbf$-bilinear form $\fm_\bbbf:(\bbbf\ot_\bbbz A)\times(\bbbf\ot_\bbbz A)\longrightarrow \bbbf$ defined  by
$$(r\ot   a,s\ot   b)_\bbbf:=rs( a, b);\;r,s\in\bbbf,\;a,b\in A.$$
Then $(\bbbf\ot_\bbbz A,\fm,1\ot R)$ is a locally finite root supersystem. Conversely, if $(\v,\fm, R)$ is a locally finite root supersystem, then for  $A:=\hbox{span}_\bbbz R,$ the map $\fm\mid_{A\times A}$ is a group bi-homomorphism with $A^0=\{0\}$ and the triple  $(A,\fm\mid_{A\times A},R)$ satisfies the conditions $(S1)-(S5)$ above.
\end{lem}
\pf  We first note that since $A^0=\{0\},$ it follows that $A$ is a torsion free abelian group and so we can identify $A$ with $1\ot A:=\{1\ot a\mid a\in A\}\sub \bbbf\ot_\bbbz A.$ We next mention that as in the proof of Lemma \ref{decom}, $R\setminus\{0\}$ can be written as a disjoint union of subsets  $R_i,$ where   $i$ runs over a nonempty  index set $ Q,$ such that
\begin{itemize}
\item for each $i\in Q,$  $R_i$ cannot be written as a disjoint union of  two orthogonal nonempty subsets,
\item  there is no $0\neq a\in \hbox{span}_\bbbz R_i$ ($i\in Q$) with $(a,\hbox{span}_\bbbz R_i)=\{0\},$
\item  for each $i\in Q,$ $R_i$ satisfies  corresponding $(S2)-(S5)$ above,
\item $(R_i,R_j)=\{0\}$ for all $i,j\in Q$ with $i\neq j.$
\end{itemize}
  So without loss of generality, we assume that $R\setminus\{0\}$ cannot be written as a disjoint union of two orthogonal  nonempty subsets. Set $\v_\bbbq:=\hbox{span}_\bbbq A\sub\v:=\bbbf\ot_\bbbz A.$
One can see that the form $\fm_\bbbf$ restricted to $\v_\bbbq$ is nondegenerate (see \cite[Lemma 1.6]{AYY}).
So  setting $\fm_\bbbq:=\fm_\bbbf\mid_{\v_\bbbq\times \v_\bbbq},$ we have $(\v_\bbbq,\fm_\bbbq,R)\in \T'_{\bbbq,\bbbf}=\T_{\bbbq,\bbbf};$ in particular, $\rre$ is  a locally finite root system in $\v_\bbbq$ and so by  Lemma \ref{real roots},
\begin{equation}\label{nondeg}
\parbox{4.4in}{
the form  $\fm_\bbbq$ restricted to $(\v_\bbbq)_{re}:=\hbox{span}_\bbbq \rre$ is nondegenerate.}
\end{equation}

To show that $(\v,\fm_\bbbf, R)\in \T'_{\bbbf,\bbbf}=\T_{\bbbf,\bbbf},$ it is enough to show that the form $\fm_\bbbf$ on $\v$ is nondegenerate. We first assume $\rim\neq\{0\}.$ Since $(\v_\bbbq,\fm_\bbbq, R)\in \T_{\bbbq,\bbbf}$ is  irreducible,  by Lemma \ref{conjugate},  for a fixed $\a^*\in\rim^\times,$ $\rre\cup\{\a^*\}$ is a spanning  set for $\v_\bbbq.$ This implies that $A\sub\hbox{span}_\bbbq(\rre\cup\{\a^*\})$ and so $\v=\hbox{span}_\bbbf A\sub\hbox{span}_\bbbf(\rre\cup\{\a^*\}).$ Suppose that for $\a_1,\ldots,\a_n\in \rre$ and $r,r_1,\ldots,r_n\in\bbbf,$ $v:=r\a^*+\sum_{i=1}^nr_i\a_i$ is an element of the radical of $\fm_\bbbf.$ Without loss of generality, we  assume $r=0,1.$ Suppose that $\{1,x_j\mid j\in J\}$ is a basis for $\bbbf$ over $\bbbq.$
For  $1\leq i\leq n,$ there are $s_i,s_i^j\in\bbbq$   such that $r_i=s_i1+\sum_j s_i^jx_j.$ Now for each $\a\in\rre^\times,$ we have
\begin{eqnarray*}
0&=&2(r\a^*+\sum_{i=1}^nr_i\a_i,\a)/(\a,\a)\\&=&\frac{2(r\a^*,\a)}{(\a,\a)}+\sum_{i=1}^n\frac{r_i2(\a_i,\a)}{(\a,\a)}\\
&=&\frac{2(r\a^*,\a)}{(\a,\a)}+\sum_{i=1}^n\frac{(s_i1+\sum_j s_i^jx_j)2(\a_i,\a)}{(\a,\a)}\\
&=&\frac{2(r\a^*,\a)}{(\a,\a)}+\sum_{i=1}^ns_i\frac{2(\a_i,\a)}{(\a,\a)}1+\sum_{i=1}^n \sum_j s_i^j\frac{2(\a_i,\a)}{(\a,\a)}x_j\\
&=&(\frac{2(r\a^*,\a)}{(\a,\a)}+\sum_{i=1}^ns_i\frac{2(\a_i,\a)}{(\a,\a)})1+\sum_{i=1}^n \sum_j s_i^j\frac{2(\a_i,\a)}{(\a,\a)}x_j.
\end{eqnarray*}
 Now as $\frac{2(r\a^*,\a)}{(\a,\a)},\frac{2(\a_i,\a)}{(\a,\a)}\in\bbbz$ $(1\leq i\leq n)$ and $\{1,x_j\mid j\in J\}$ is a basis for the $\bbbq$-vector space $\bbbf,$ for each $j\in J$ and  $\a\in\rre^\times,$ we have  $\sum_{i=1}^ns_i^j\frac{2(\a_i,\a)}{(\a,\a)}=0.$  Therefore, for all $j\in J$ and $\a\in\rre^\times,$ $ (\sum_{i=1}^ns_i^j \a_i,\a)=0.$  This together with (\ref{nondeg}) implies that for $j\in J,$  $\sum_{i=1}^ns_i^j \a_i=0$ and so $$v=r\a^*+\sum_{i=1}^nr_i\a_i=r\a^*+\sum_{i=1}^ns_i\a_i\in\hbox{span}_\bbbq(\rre\cup\{\a^*\}).$$ Thus $v$ is an element of the radical of $\fm_\bbbq.$ Now as $\fm_\bbbq$ is nondegenerate,  we get $v=0$ and so we are done in this case. Using the same argument as above, one can get the result if $\rim=\{0\}.$ The reverse implication is easy to see.\qed

\section{classification of irreducible locally finite root supersystems}
In this section, we give the classification of irreducible locally finite root supersystems. In \cite{serg}, the author gives the classification of generalized root systems (finite root supersystems in our sense). To start the classification, she uses the fact that each finite dimensional subspace  $\w$ of  a vector space equipped with a  symmetric bilinear form has an orthogonal complement provided that the form restricted to $\w$ is nondegenerate and works with a certain  orthogonal decomposition. As we are working with possibly infinite dimensional vector spaces, the above stated fact cannot be used in our case. We divided irreducible locally finite root supersystems in two disjoint classes and classify the elements of  each class separately.


\begin{deft}{\rm
For an irreducible  locally finite root supersystem {\small $(\v,\fm,R)$}, we say $R$ is of {\it real type} if $\v=\v_{re}=\hbox{span}_\bbbf \rre;$ otherwise we say it is of {\it imaginary type}.}
\end{deft}
\begin{Example}\label{im-sys}
{\rm (1) Suppose that $T$ is an index set with $|T|\geq 2.$
Take $\u$ to be a vector space with a basis $\{\ep_t\mid t\in T\}$ and consider  the bilinear form $$\fm:\u\times \u\longrightarrow \bbbf;\;\;(\ep_t,\ep_{t'})\mapsto \d_{t,t'} \;\;(t,t'\in T)$$ on $\u.$
Take   $S:=\{\pm\ep_t\pm\ep_{t'}\mid t,t'\in T\}$ to be the locally finite root system of type $C_T$ in $\u.$ Consider a one dimensional vector space $\bbbf\a^*$ and set $\v:=\bbbf\a^*\op\u.$  Fix  $t_0\in T.$ Extend the form $\fm$ to a symmetric bilinear  form on $\v$ denoted again by $\fm$ and defined by $$(\a^*,\a^*):=0,\;(\a^*,\ep_{t_0}):=1,\;(\a^*,\ep_t):=0;\; t\in T\setminus\{t_0\}.$$ Set $R:=S\cup\pm\{\a^*,\a^*-2\ep_{t_0},\a^*-(\ep_{t_0}\pm\ep_t)\mid t\in T\setminus\{t_0\}\}.$ Then $(\v,\fm,R)$ is an irreducible locally finite root supersystem of imaginary type which we refer to as type $\dot C(0,T).$ If $T$ is an infinite set and  $\{T_k\mid k\in K\}$ is the class of all finite subsets of $T$ containing $t_0$ with cardinal number greater than 1, then $R=\cup_{k\in K}R_k,$ where $$R_k:=\{\pm\ep_t\pm\ep_{t'}\mid t,t'\in T_k\}\cup\pm\{\a^*,\a^*-2\ep_{t_0},\a^*-(\ep_{t_0}\pm\ep_t)\mid t\in T_k\setminus\{t_0\}\},$$ in other words, $R$ is a direct union of finite root supersystems $\dot C(0, T_k),$ $k\in K.$
 }

{\rm (2) Suppose that $T$ is an index set with $|T|\geq2$.
Take $\u$ to be a vector space with a basis $\{\ep_t\mid t\in T\}$ and consider  the bilinear form $$\fm:\u\times \u\longrightarrow \bbbf;\;\;(\ep_t,\ep_{t'})\mapsto \d_{t,t'} \;\;(t,t'\in T)$$ on $\u.$
Take   $S:=\{\ep_t-\ep_{t'}\mid t,t'\in T\}$ to be the locally finite root system of type $\dot A_T$ in $\u':=\hbox{span}_\bbbf S.$
  Consider  a one dimensional vector space $\bbbf\a^*$ and set $\v:=\bbbf\a^*\op\u'.$  Fix an element $t_0\in T.$ Extend the form $\fm$ to a symmetric bilinear  form on $\v$ denoted again by $\fm$ and defined by $$(\a^*,\a^*):=0,\;(\a^*,\ep_t-\ep_{t_0}):=-1;\; t,t'\in T\setminus\{t_0\}.$$ Next set $R:=S\cup\pm\{\a^*,\a^*-(\ep_{t_0}-\ep_t)\mid t\in T\setminus\{t_0\}\}.$ Then $(\v,\fm,R)$ is an  irreducible  locally finite root supersystem of imaginary type which we refer to as type $\dot A(0,T).$ If $T$ is an infinite set and  $\{T_k\mid k\in K\}$ is the class of all finite subsets of $T$ containing $t_0$ with cardinal number greater than 1, then for $$R_k:=\{\pm(\ep_t-\ep_{t'})\mid t,t'\in T_k\}\cup\pm\{\a^*,\a^*-(\ep_{t_0}-\ep_t)\mid t\in T_k\setminus\{t_0\}\},$$ we have $R=\cup_{k\in K}R_k.$ This means that  $R$ is a direct union of finite root supersystems $R_k$ ($k\in K$) of type $\dot A(0,T_k).$
}

{\rm(3)
 Suppose that $T,P$ are two index sets of cardinal numbers greater than 1  such that if $S,T$ are both finite, then $|T|\not=|P|.$ Suppose that  $\v'$ is a vector space with a basis $\{\ep_t,\d_{p}\mid t\in T,\; p\in P\}.$ We equip $\v'$ with a symmetric bilinear form $$\fm:\v'\times\v'\longrightarrow \bbbf;\;(\ep_t,\ep_{t'})=\d_{t,t'},\;(\d_p,\d_{p'})=-\d_{p,p'},(\d_p,\ep_t)=0.$$  Take    $S_1:=\{\ep_t-\ep_{t'}\mid t,t'\in T\}$ which is a locally finite root system of type $\dot A_T$ in $\u_1:=\hbox{span}_\bbbf\{\ep_t-\ep_{t'}\mid t,t'\in T\}$ and $S_2:=\{\d_p-\d_{p'}\mid p,p'\in P\}$ which is a locally finite root system of type $\dot A_P$ in $\u_2:=\hbox{span}_{\bbbf}\{\d_p-\d_{p'}\mid p,p'\in P\}.$ Now set $\v:=\hbox{span}_\bbbf\{\ep_t-\d_p\mid t\in T,p\in P\}+\u_1+\u_2$  and take $R:=S_1\cup S_2\cup \pm\{\ep_t-\d_p\mid t\in T,p\in P\}.$ Then $(\v,\fm\mid_{\v\times\v},R)$ is an  irreducible  locally finite root supersystem of imaginary type with $\rim=\{0\}\cup\pm\{\ep_t-\d_p\mid t\in T,p\in P\}.$ We refer to this locally finite root supersystem  as a locally finite root supersystem of type  $\dot A(T,P).$ Now suppose that  at least one of $T$ and $P$ are infinite. Fix $t_0\in T,p_0\in P$ and  suppose that $\{T_k\mid k\in K\}$
(resp. $\{P_m\mid m\in M\}$) is the class of all finite subsets of $T$ (resp. $P$) containing $t_0$ (resp. $p_0$) with cardinal number greater than 1. Take $\Lam:=\{(k,m)\in K\times M\mid |T_k|\neq|P_m|\}.$  Now for $(k,m)\in \Lam,$ set{\small$$R_{(k,m)}:=\{\ep_t-\ep_{t'}\mid t,t'\in T_k\}\cup\{\d_p-\d_{p'}\mid p,p'\in P_m\}\cup\{\pm(\ep_t-\d_p)\mid t\in T_k,p\in P_m\}.$$}

\noindent Then $R_{(k,m)}$ is a finite root supersystem  of type $\dot A(T_k,P_m)$ in its $\bbbf$-span and $R$ is the direct union of $\{R_{(k,m)}\mid (k,m)\in \Lam\}.$

}
\end{Example}

\begin{Example}
{\rm (1)
 Suppose that $S_1:=\{0,\pm2\dot\ep_1\}$ and $S_2:=\{0,\pm2\dot\d_1\}$ are finite root systems of type $A_1$ in $\hbox{span}_\bbbf\dot\ep_1$ and $\hbox{span}_\bbbf\dot\d_1$ respectively.  Normalize the forms on $\hbox{span}_\bbbf\dot\ep_1$ and $\hbox{span}_\bbbf\dot\d_1$ such that $(\dot\ep_1,\dot\ep_1)=-(\dot\d_1,\dot\d_1)$ and extend them to a form $\fm$ on $\hbox{span}_\bbbf\{\dot\ep_1,\dot\d_1\}$ with $(\dot\ep_1,\dot\d_1):=0.$ Set $R:=S_1\cup S_2\cup \{\pm\dot\ep_1\pm\dot\d_1\}.$ Then $R$ is an irreducible locally finite root supersystem of real type in $\hbox{span}_\bbbf\{\dot\ep_1,\dot\d_1\}$ which  we refer to as type $A(1,1).$}

{\rm (2) Suppose that  $S_1:=\{0,\pm2\dot\ep_1\}$ and $S_2:=\{0,\pm 2\d_t,\pm\d_t\pm\d_{t'}\mid t,t'\in T,\;t\neq t'\}$ are locally  finite root systems of type $A_1$ and $C_T$ respectively, where $T$ is an index set with $|T|\geq2.$ Normalize the forms on $\hbox{span}_\bbbf\dot\ep_1$ and $\hbox{span}_\bbbf\{\d_t\mid t\in T\}$ such that $(\dot\ep_1,\dot\ep_1)=-(\d_t,\d_{t});$ $t\in T.$ Set $R:=S_1\cup S_2\cup \{\pm\dot\ep_1\pm\d_t\mid t\in T\}.$ Then $R$ is an irreducible locally finite root supersystem of real type in $\hbox{span}_\bbbf\{\dot\ep_1,\dot\d_t\mid t\in T\}$ which  we refer to as type $C(1,T).$} \end{Example}

\begin{lem}\label{finite}
Suppose that $(\v,\fm,R)$ is  a locally finite root supersystem  with $\rre\neq \{0\}$ and that  $R_{re}=\op_{i\in I}\rre^i$ is the decomposition of the locally finite root system $\rre$ into irreducible  subsystems. Set $\v_{re}^i:=\hbox{span}_\bbbf\rre^i,$ $i\in I,$ and suppose $\v=\op_{i\in I}\v_{re}^i.$  For $i\in I,$ take $p_i:\v\longrightarrow \v_{re}^i$ to be the projection map on $\v_{re}^i,$  then we have  the following  statements:

(i) For   $i\in I$ and  $\a\in R,$ we have  $p_i(\a)\in\hbox{span}_\bbbq \rre^i;$ in particular, we have either $p_i(\a)=0$ or  $(p_i(\a),p_i(\a))\neq0.$

(ii) If $\rre$ is a finite root system, then $R$ is a finite root supersystem.
\end{lem}
\pf  $(i)$ It follows using the same argument as in \cite[Corollary 1.7]{serg} and Lemma \ref{inv-form}.

$(ii)$ We know from Lemma \ref{decom} that $R$ is a direct sum of nonzero irreducible sub-supersystems, say $R=\op_{k\in K}R_k$ in  which $K$ is an index set. We claim that for all $k\in K,$ $(R_k)_{re}\neq \{0\}.$ For this, take $K_1$ to be the subset of $K$ consisting of the indices $k$ for which $(R_k)_{re}\neq\{0\}$ and set $K_2:=K\setminus K_1.$ Then we have $R=(\op_{k\in K_1}R_k)\op(\op_{k\in K_2}R_{k})$ and that $(\op_{t\in K_1}R_t,\op_{t\in K_2}R_t)=\{0\}.$ Now suppose $K_2\neq \emptyset$ and fix $k\in K_2.$ Let $0\neq \d\in R_k,$ then $\d=\sum_{i\in I}p_i(\d).$ Since $\d\neq 0,$ there is $j\in I$ with  $p_j(\d)\neq 0.$ So by part ($i$), $(\d,p_j(\d))=(p_j(\d),p_j(\d))\neq 0.$ Therefore $(\d,\v_{re}^j)\neq \{0\}$ and so  $(\d,\rre^j)\neq \{0\}.$ Now as $\rre^j\sub\rre=\cup_{k\in K_1}(R_k)_{re},$ one can find $k'\in K_1$ with $(\d,(R_{k'})_{re})\neq\{0\}.$ So  $\{0\}\neq (\d, R_{k'})\in (\op_{t\in K_2}R_t,\op_{t\in K_1}R_t)=\{0\}$ which is a contradiction. This means that $K_2=\emptyset$ and $K=K_1.$ Now as  $\rre=\uplus_{k\in K_1}(R_k)_{re}$ is a finite root system, we have $|K_1|<\infty.$ Thus to complete the proof, it is enough to show that $R_k$ is a finite root supersystem  for all $k\in K.$ Since $(R_k)_{re}$ is a finite root system,   its Weyl group is a finite group. Now using Proposition \ref{conjugate}($ii$), we get that $|(R_k)_{ns}|<\infty$ and so we are done.
\qed

\medskip

{\it From now till the end of this section,  we assume $(\v,\fm,R)$ is an irreducible locally finite root supersystem with $\rre\neq \{0\}.$}  We suppose $\rre=\op_{i\in I}\rre^i$ is the decomposition of  the locally finite root system $\rre$ into irreducible subsystems. For $i\in I,$ we set $$\v_{re}^i:=\hbox{span}_\bbbf\rre^i.$$

\begin{lem}\label{tame}
If $\d\in\rim^\times$ and $i\in I,$
then we have $(\d,\rre^i)\neq\{0\}.$
\end{lem}
\pf Since  $R=(\rre^i\setminus\{0\})\uplus(\cup_{j\in I\setminus\{i\}}\rre^j\cup\rim),$ the  irreducibility of $R$ implies that $(\rre^i,\rim)\neq \{0\}.$ Now as $\rre^i$ is invariant under the Weyl group, we are done using Lemma \ref{conjugate}($ii$).
\qed

\medskip

\subsection{Imaginary type}
In this subsection, \emph{we assume $(\v,\fm,R)$ is of imaginary type}. So $\v\neq\v_{re}.$ Fix  $\a^*\in\rim^{\times},$ then by Lemma \ref{conjugate}($ii$), $$\v=\bbbf \a^*\op \v_{re}.$$
 We suppose  $$p_*:\v\longrightarrow \bbbf\a^*$$ is the canonical projection map of $\v$ on $\bbbf\a^*$  with respect to the decomposition $\v=\bbbf \a^*\op \v_{re}.$  For $\a\in R,$ take $p_\a\in\bbbf$ to be defined by $p_*(\a)=p_\a\a^*.$

\begin{lem}\label{alphastar}
For  $\a\in \rre,$ $p_\a=0$ and for $\a\in\rim\setminus\{0\},$  $p_\a=\pm1.$
\end{lem}
\pf The statement is immediate for $\a\in\rre.$ To prove the statement  for nonsingular roots, we  show  that for each two elements $\a,\b$ of $\rim^\times,$ $p_\a=\pm p_\b.$ Suppose that $\a,\b\in \rim^\times,$ we say $\a\sim\b$  if $p_\a=\pm p_\b.$ This defines an equivalence relation on $\rim^\times.$ So $\rim^\times$ is the disjoint union of equivalence classes, say $\rim^\times=\uplus_{t\in T}\rim^{t}.$  Suppose that  $t,t'\in T$ with $t\neq t',$  $\a\in\rim^t$ and $\a'\in\rim^{t'}.$ If $(\a,\a')\neq0,$ then since $R$ is a locally finite root supersystem, there is $k\in\{\pm1\}$ such that $\a+k\a'\in R.$ Since $\a,\a'\in\rim$ and $(\a,\a')\neq 0,$ we get that $\a+k\a'\in\rre$ and so $p_*(\a+k\a')=0.$ This in turn implies that $p_*(\a)=-kp_*(\a');$ in particular $\a\sim\a'$ which is  a contradiction. Therefore $(\rim^t,\rim^{t'})=\{0\}.$ Also for $t\in T,$ $\d\in \rim^t$ and  $\a\in\rre^\times,$ we have $r_\a\d=\d-2\frac{(\a,\d)}{(\a,\a)}\a,$ so $p_*(\d)=p_*(r_\a(\d)).$ Therefore  $r_\a(\d)\in\rim^t$  and so $\rim^t$ is invariant under the  Weyl group. Now using Lemma \ref{general}, we get that for  all $\a,\b\in\rim^\times,$ $\a\sim \b;$ in particular, $\a\sim\a^*$ for all $\a\in\rim^\times.$ So $p_\a=\pm p_{\a^*}=\pm1.$\qed

\begin{lem}\label{w-alpha-star} For $i\in I,$ take $\w_i$ to be the subgroup of $\w$ generated by $\{r_\a\mid \a\in \rre^i\setminus\{0\}\}.$ We have the following statements:
\\
(i) Suppose that $i\in I$ and  $w\in\w_i,$ then  we have $w\a^*-\a^*\in \rre^i.$
\\
(ii)
Suppose that $w\in\w,$ then there are distinct elements   $i_1,\ldots,i_t$ of $ I$ and $w_{1}\in \w_{i_1},\ldots,w_t\in\w_{i_t}$  such that  $w\a^*=w_1\ldots w_t\a^*.$  Moreover, $w\a^*=\a^*+(w_1\a^*-\a^*)+\cdots+(w_t\a^*-\a^*).$
\end{lem}
\pf  $(i)$ Suppose that  $w:=r_{\a_n}\cdots r_{\a_1},$ for some $\a_1,\ldots,\a_n\in \rre^i\setminus\{0\}.$ We first  use induction on $n$ to prove that $w\a^*-\a^*\in \hbox{span}_\bbbq\rre^i.$ If $n=1,$ we have $w\a^*=r_{\a_1}\a^*=\a^*-2\frac{(\a_1,\a^*)}{(\a_1,\a_1)}\a_1$ and so we are done as $2\frac{(\a_1,\a^*)}{(\a_1,\a_1)}\in\bbbz.$
Now suppose $n>1$  and that the result  holds for $n-1.$ Set $w_1:=r_{\a_{n-1}}\cdots r_{\a_1},$ then by the induction hypothesis, we get that  $w_1\a^*-\a^*\in \hbox{span}_\bbbq\rre^i.$ So
\begin{eqnarray}
w\a^*=r_{\a_n}w_1\a^*\in r_{\a_n}(\a^*+\hbox{span}_\bbbq\rre^i)&=&r_{\a_n}(\a^*)+r_{\a_n}(\hbox{span}_\bbbq\rre^i)\nonumber\\&\sub& \a^*-2\frac{(\a_n,\a^*)}{(\a_n,\a_n)}\a_n+\hbox{span}_\bbbq\rre^i\label{final2}\\&\sub & \a^*+\hbox{span}_\bbbq\rre^i\nonumber
\end{eqnarray}
and so $w\a^*-\a^*\in \hbox{span}_\bbbq\rre^i.$ Now by Lemmas \ref{inv-form} and \ref{real roots}($a$), since $w\a^*-\a^*\in\hbox{span}_\bbbq\rre^i,$ either $w\a^*=\a^*$ or $(w\a^*-\a^*,w\a^*-\a^*)\neq0.$ If  $(w\a^*-\a^*,w\a^*-\a^*)\neq0,$ then $(w\a^*,\a^*)\neq 0$ and so by the definition of  a locally finite root supersystem, $ \{w\a^*\pm \a^*\}\cap R\neq\emptyset.$ But by (\ref{final2}), we have $p_*(w\a^*)=p_*(\a^*)$ which together with    Lemma   \ref{alphastar} implies that  $w\a^*+\a^*\not\in R.$ Therefore  $w\a^*-\a^*\in R$ and so $w\a^*-\a^*\in R_{re}\cap\v^i_{re}=\rre^i.$ This completes the proof.

($ii$) The first assertion follows from the fact that for $\a,\b\in \rre^\times$ with $(\a,\b)= 0,$ we have $r_\a r_\b=r_\b r_\a.$
 Now suppose that          $i_1,\ldots,i_t$  are distinct elements of $ I$ and $w_{1}\in \w_{i_1},\ldots,w_t\in\w_{i_t}$  such that  $w\a^*=w_1\ldots w_t\a^*.$ We know from part ($i$) that for $1\leq j\leq t,$ $w_j\a^*-\a^*\in \rre^{i_j},$ so for $1\leq j\neq r\leq t,$ we have $w_r(w_j\a^*-\a^*)=w_j\a^*-\a^*.$ Now using  induction on $t,$ one  gets that $$w\a^*=\a^*+(w_1\a^*-\a^*)+\cdots+(w_t\a^*-\a^*).$$ This completes the proof.\qed

\medskip

Consider the decomposition $\v=\bbbf\a^*\op\sum_{i\in I}\v_{re}^i$ and assume for $i\in I,$ $p_i$ is the projection map of $\v$ on $\v_{re}^i$ with respect to this decomposition. For $\a\in R,$ define $$\hbox{supp}(\a):=\{i\in I\mid p_i(\a)\neq 0\}.$$ Using Lemma \ref{w-alpha-star}, one can easily verify the following corollary.
\begin{cor}
 \label{note} For $\a\in \rre^\times,$ $|\supp(\a)|=1.$ Also for $\d\in\rim$ and  $i\in I,$ $p_i(\d)\in \rre^i.$
\end{cor}
\begin{lem}
\label{supp}
Suppose that $\d,\gamma\in\rim\setminus\{0\},$ then there is $\eta\in\rim$ such that $\supp(\d)\cup\supp(\gamma)\sub\supp(\eta).$
\end{lem}
\pf  If $\supp(\d)\sub\supp(\gamma)$ or $\supp(\gamma)\sub\supp(\d),$ then there is nothing to prove, so we assume $\supp(\d)\not\sub\supp(\gamma)$  as well as $\supp(\gamma)\not\sub\supp(\d).$ Suppose that $\supp(\d)\cap\supp(\gamma)=\{i_1,\ldots,i_t\}$  and note that this can be the empty set. We make a convention that in what follows we remove the expressions involving  $i_1,\ldots,i_t,$  if $\supp(\d)\cap\supp(\gamma)=\emptyset.$ Assume that $\supp(\d)=\{i_1,\ldots,i_t,j_1,\ldots,j_n\}$ and $\supp(\gamma)=\{i_1,\ldots,i_t,k_1,\ldots,k_m\}.$ So by Lemma \ref{w-alpha-star} and Proposition \ref{conjugate}($ii$), there are $w_1,w'_1\in \w_{i_1},\ldots,w_t,w'_t\in \w_{i_t},$  $u_1\in\w_{j_1},\ldots,u_n\in\w_{j_n},$ $v_1\in \w_{k_1},\ldots,v_m\in\w_{k_m}$  and $r,s\in\{\pm 1\}$ such that
$$\begin{array}{l}
u_1\a^*-\a^*\neq0,\ldots,u_{n}\a^*-\a^*\neq 0,\\
v_1\a^*-\a^*\neq0,\ldots,v_{m}\a^*-\a^*\neq 0
\end{array}$$
and {\small $$\begin{array}{l}r\d=\a^*+(w_1\a^*-\a^*)+\cdots+(w_{t}\a^*-\a^*)+(u_1\a^*-\a^*)+\cdots+(u_{n}\a^*-\a^*),\\
s\gamma=\a^*+(w'_1\a^*-\a^*)+\cdots+(w'_{t}\a^*-\a^*)+(v_1\a^*-\a^*)+\cdots+(v_{m}\a^*-\a^*).\end{array}$$}
Now set $\eta:=w_1\cdots w_t u_1\cdots u_n v_1\cdots v_m(\a^*),$ then we have
{\small \begin{eqnarray*}
\eta&=&\a^*+(w_1\a^*-\a^*)+\cdots+(w_{t}\a^*-\a^*)+(u_1\a^*-\a^*)+\cdots+(u_{n}\a^*-\a^*)\\
&+&(v_1\a^*-\a^*)+\cdots+(v_{m}\a^*-\a^*).
\end{eqnarray*}}
This completes the proof.\qed

\begin{lem}
\label{not-decom}
(i) Suppose that $\a\in \rre^\times$ and $r\in\{\pm1\}$ are such that $\a+r\a^*\in R,$ then $\a+r\a^*\in\rim,$ in particular $2(\a,\a^*)/(\a,\a)=-r.$

(ii) For each $i\in I,$ there is $\d\in\rim$ such that $i\in \supp(\d).$

(iii) If $i\in I$ and $w\in\w,$ then $(w\a^*,\rre^i)\neq \{0\}.$
\end{lem}

\pf $(i)$ If $\gamma:=\a+r\a^*\in\rre,$ then $0=r\a^*+\a-\gamma\in\bbbf\a^*\op\v_{re}$ which is a contradiction. Therefore, $\gamma\in\rim,$ and so   we have  $0=(\gamma,\gamma)=(\a+r\a^*,\a+r\a^*)=(\a,\a)+2r(\a,\a^*),$ This in turn implies that  $2(\a,\a^*)/(\a,\a)=-r.$

$(ii)$ Let $i\in I,$ we know from Lemma \ref{tame} that  $(\rre^i,\a^*)\neq\{0\},$ so one finds $\a\in \rre^i$ such that $(\a,\a^*)\neq 0.$ Therefore there is $r\in\{\pm 1\}$ such that $\a+r\a^*\in R.$ Using part $(i),$ we get that $\d:=\a+r\a^*\in\rim$ and $i\in \supp(\d).$

$(iii)$ It is immediate as by Lemma \ref{tame}, $(\rre^i,\a^*)\neq\{0\}$ and the form is $\w$-invariant.
\qed

\begin{Pro}
\label{cardinality} With the same notation as above,
$|I|\leq 2.$
\end{Pro}
\pf Suppose  to the contrary that $|I|>2.$ Fix distinct elements $i_1,i_2,i_3\in I.$ Using Lemma \ref{not-decom}($ii$), we pick $\d_j\in\rim$ $(1\leq j\leq 3)$ such that $i_j\in\supp(\d_j).$ We next use Lemma \ref{supp} to fix $\d\in\rim$ such that $i_1,i_2,i_3\in \supp(\d).$ For $1\leq j\leq 3,$ by Lemma \ref{w-alpha-star} and Proposition \ref{conjugate}($ii$), there is $w_j\in\w_{i_j}$  such that $w_j\a^*-\a^*=\pm p_{i_j}(\d)\neq 0.$ Now we can use Lemma \ref{not-decom}($iii$) to find $\gamma_j\in \rre^{i_j}$ such that $(\gamma_j,w_j\a^*)\neq 0.$
For distinct indices $i,j\in\{1,2,3\},$ we set $$\a_{i,j}:=r_{\gamma_j}r_{\gamma_i}\d=\d-\frac{2(\gamma_j,\d)}{(\gamma_j,\gamma_j)}\gamma_j-\frac{2(\gamma_i,\d)}{(\gamma_i,\gamma_i)}\gamma_i.$$
Using the same argument as in \cite[Propoisition 2.6(1)]{serg},
%
one can find $i,j\in\{1,2,3\}$ with $i\neq j$ such that $(\d,\a_{i,j})\neq 0.$ Therefore  there is $r_{i,j}\in\{\pm1\}$ such that $\b_{i,j}=\a_{i,j}+r_{i,j}\d\in R.$ If $r_{i,j}=1,$ then  $\b_{i,j}=\a_{i,j}+\d\in R.$ But $p_*(\b_{i,j})=p_*(2\d-\frac{2(\d,\gamma_i)}{(\gamma_i,\gamma_i)}\gamma_i-\frac{2(\d,\gamma_j)}{(\gamma_j,\gamma_j)}\gamma_j)=2p_*(\d)$ which is a contradiction using Lemma \ref{alphastar}. Therefore $r_{i,j}=-1$ and so $\b_{i,j}=\a_{i,j}-\d=-\frac{2(\d,\gamma_i)}{(\gamma_i,\gamma_i)}\gamma_i-\frac{2(\d,\gamma_j)}{(\gamma_j,\gamma_j)}\gamma_j\in R.$ Since  $(\a_{i,j}-\d,\a_{i,j}-\d)=-2(\d,\a_{i,j})\neq0,$ we have  $-\frac{2(\d,\gamma_i)}{(\gamma_i,\gamma_i)}\gamma_i-\frac{2(\d,\gamma_j)}{(\gamma_j,\gamma_j)}\gamma_j\in\rre^\times$ which is a contradiction using Corollary \ref{note}. This completes the proof.\qed
\begin{Pro} \label{main}Suppose that $i\in I.$ Then
\begin{eqnarray*}\aa&:=&\{\a\in\rre^i\setminus\{0\}\mid2(\a,\a^*)/(\a,\a)=1\}\\
&=&\{\a\in\rre^i\setminus\{0\}\mid(\a,\a^*)=(\a,\a)/2\}\end{eqnarray*} is a nonempty set. Moreover, for $\a,\b\in\aa,$ we have $\a-\b\in\rre^i.$
\end{Pro}
\pf We know from Lemma \ref{tame} that $(\a^*,\rre^i)\neq \{0\}.$ Fix $\a\in \rre^i$ such that $(\a,\a^*)\neq 0.$ So there is $r\in\{\pm 1\}$ such that $\a+r\a^*\in R.$ By Lemma \ref{not-decom}($i$), $2(\a,\a^*)/(\a,\a)=-r.$ Therefore,  either $\a\in\aa$ or $-\a\in \aa.$
Now suppose   $\a,\b\in\aa,$ then   we have  \begin{equation}\label{subtract}r_\a\a^*=\a^*-2\frac{(\a,\a^*)}{(\a,\a)}\a=\a^*-\a\andd r_\b\a^*=\a^*-2\frac{(\b,\a^*)}{(\b,\b)}\b=\a^*-\b.\end{equation}
This implies that     \begin{equation}\label{conj}\a-\b=r_\b\a^*-r_\a\a^*=r_\a(r_\a r_\b\a^*-\a^*).\end{equation}
So \begin{equation}\label{equal1}\a=\b \hbox{ if and only if } r_\a r_\b\a^*=\a^*.\end{equation}On the other hand by (\ref{subtract}), we have
\begin{equation*}\label{compute}r_\a r_\b\a^*-\a^*=r_\a(\a^*-\b)-\a^*=\a^*-\a-r_\a\b-\a^*=-\a-r_\a\b\in \hbox{span}_\bbbq \rre^i.\end{equation*}
This together with (\ref{equal1}) and  Lemma \ref{inv-form} implies that \begin{equation}\label{nonzero}-2(r_\a r_\b\a^*,\a^*)=(r_\a r_\b\a^*-\a^*,r_\a r_\b\a^*-\a^*)\neq 0;\;\;\a,\b\in\aa;\;\a\neq \b.\end{equation}  Now if $\a,\b$ are two distinct elements of $\aa,$ (\ref{nonzero}) implies that either $r_\a r_\b\a^*+\a^*\in R$ or $r_\a r_\b\a^*-\a^*\in R;$ but by Lemma  \ref{alphastar}, one knows that $r_\a r_\b\a^*+\a^*\not\in R,$ so  $r_\a r_\b\a^*-\a^*\in R\cap \hbox{span}_\bbbq \rre^i=\rre^i.$  This together with  (\ref{conj}) completes the proof.
\qed

\begin{Pro}
\label{pre}Suppose that $i\in I.$ Consider Proposition \ref{main} and set $A:=\pm \aa=\{\a\in\rre^i\mid (\a,\a^*)\neq 0\}.$

(i)  If $\rre^i$ is a locally finite root system of type $\dot A_T$ for an index set $T$ with $|T|\geq 2,$ say $\rre^i=\{\ep_r-\ep_s\mid r,s\in T\},$ then there is $t_0\in T$ such that $A=\{\pm(\ep_{t_0}-\ep_t)\mid t\in T\setminus\{t_0\}\}.$  In particular, for $r,s\in T\setminus\{t_0\},$ $(\a^*,\ep_{t_0}-\ep_r)=(\a^*,\ep_{t_0}-\ep_s).$

(ii) If $\rre^i$ is a locally finite root system of type $C_T$ for an index set $T$ with $|T|\geq 2,$ say $\rre^i=\{\pm(\ep_r\pm\ep_s)\mid r,s\in T\},$ then there is $t_0\in T$ such that $A=\{\pm 2\ep_{t_0},\pm(\ep_{t_0}\pm\ep_t)\mid t\in T\setminus\{t_0\}\}.$ In particular, for $t\in T\setminus\{t_0\},$ $(\a^*,\ep_t)=0.$
\end{Pro}
\pf $(i)$ We know from Proposition \ref{main} that  $A\neq \emptyset.$ So there is nothing to prove if $\rre^i$ is of rank 1. Next suppose   $\rre^i$ is of rank greater than 1. If  there are distinct elements $r,s,t\in T$ with  $\ep_r-\ep_s\in A$ and  $(\a^*,\ep_r-\ep_t)=(\a^*,\ep_s-\ep_t)=0,$ then we have  $(\a^*,\ep_r-\ep_s)=0$ which is a contradiction. Therefore, on concludes that
  if $r,s\in T$ with $\ep_r-\ep_s\in A,$ then
for $t\in T,$ either $\ep_r-\ep_t\in A,$ or $\ep_s-\ep_t\in A.$
This together with Proposition \ref{main} completes  the proof.

$(ii)$ We know that $A\neq \emptyset.$  If $A_{lg}:=A\cap (\rre^i)_{lg}=\emptyset,$ then for all $r\in T,$ $(\ep_{r},\a^*)=0$ which in turn implies that $(\pm\ep_r\pm\ep_s,\a^*)=0$ for all $r,s\in T.$ In other words, $A=\emptyset,$ a contradiction. So  $A_{lg}\neq \emptyset.$ We claim that  $|A_{lg}|=1.$  Suppose that $t_0\in T$ is such that $\a:=2\ep_{t_0}\in A_{lg}.$ If   $\b:=2\ep_t\in A$ for some $t\in T\setminus\{t_0\},$ then there are $k,k'\in\{\pm1\}$ with  $\d:=2\ep_{t_0}+k\a^*,\gamma:=2\ep_t+k'\a^*\in\rim.$ Now we have  using Lemma \ref{not-decom}($i$) that $(\d,\gamma)=2k'(\ep_{t_0},\a^*)+2k(\ep_{t},\a^*)=-2(\ep_{t_0},\ep_{t_0})-2(\ep_{t},\ep_{t})=-4(\ep_{t_0},\ep_{t_0})\neq 0.$ This implies that either $\d-\gamma\in R$ or $\d+\gamma\in R;$ but this is  a contradiction by  Lemma \ref{alphastar}. So $A_{lg}=\{\pm2\ep_{t_0}\};$  in particular, $$(\ep_t,\a^*)=0;\;\;\hbox{ for all $t\in T\setminus\{t_0\}.$}$$
Now fixing $t\in T\setminus\{t_0\},$ we have \begin{eqnarray*}2(\ep_{t_0}\pm\ep_t,\a^*)=2(\ep_{t_0},\a^*)\neq0;\end{eqnarray*}
in other words $\pm\ep_{t_0}\pm\ep_t\in A$ for all $t\in T\setminus\{t_0\}.$ These together with Proposition \ref{main} complete the proof.
\qed
\begin{Thm}
\label{type}  For $i\in I,$ $\rre^i$ is of type $A$ or $C$ and if $|I|=2,$ $\rre^i$ is of type $A.$
\end{Thm}
\pf Suppose that $i\in I.$ Take $\aa$ to be defined as in Proposition \ref{main}. We carry out the proof through the following steps:
\smallskip


\noindent{\bf Step 1.} $\rre^i$ is not of types $B,D,BC:$ Take $$\aa_{sh}:=(\rre^i)_{sh}\cap\aa, \;\; \aa_{lg}:=(\rre^i)_{lg}\cap\aa\andd \aa_{ex}:=(\rre^i)_{ex}\cap\aa.$$

We first suppose that $\rre^i$ is of type $B_T$ for an  index set $T$ with $|T|\geq3.$ We may assume $\rre^i=\{0,\pm\ep_r,\pm\ep_r\pm\ep_s\mid r,s\in T;\; r\neq s\}.$
We  claim that $\aa_{sh}\neq \emptyset.$ Indeed, if $\aa_{sh}=\emptyset,$ then  for all $i\in T,$ $\ep_{i}\not\in \aa_{sh},$ so $(\a^*,\ep_i)=0$ by Lemma \ref{not-decom}($i$) and the fifth condition of a locally finite root supersystem. This implies that for all $i,j\in T,$ $(\pm\ep_i\pm\ep_j,\a^*)=0,$ i.e., $\aa_{lg}=\emptyset.$ So $\aa=\emptyset$ which  contradicts Proposition \ref{main}. Therefore, we have $\aa_{sh}\neq \emptyset.$ Fix   $i_0\in T$ and $p\in\{\pm1\}$ with $p\ep_{i_0}\in\aa,$ then $\a^*-p\ep_{i_0}\in \rim$ and $2(\a^*,p\ep_{i_0})=(\ep_{i_0},\ep_{i_0}).$ If there is $j\in T$ with  $(\a^*,\ep_j)=0,$ then we have $$(p\ep_{i_0}-\a^*,p\ep_{i_0}-\ep_j)=(p\ep_{i_0},p\ep_{i_0})-(p\ep_{i_0},p\ep_{i_0})/2=(p\ep_{i_0},p\ep_{i_0})/2\neq 0.$$
Therefore there is $q\in\{\pm1\}$ with $\gamma:=p\ep_{i_0}-\a^*+q(p\ep_{i_0}-\ep_j)\in R.$ It follows from Corollary \ref{note} that $q=-1$ and so $\gamma=\ep_j-\a^*\in R.$ Now using Lemma \ref{not-decom}, we get that $(\ep_j,\a^*)=(\ep_j,\ep_j)/2\neq 0$ which is a contradiction. Thus for all $j\in T,$ $(\a^*,\ep_j)\neq 0,$ in particular,  \begin{equation}\label{a short}\aa_{sh}=\{s_t\ep_t\mid t\in T\}\;\;\;\;\hbox{for some $s_t\in\{\pm1\}$ ($t\in T$).}\end{equation} Therefore  for $r,t\in T$ with $r\neq t,$ we have $$(s_t\ep_t,\a^*)=(s_t\ep_t,s_t\ep_t)/2=(s_r\ep_r,s_r\ep_r)/2=(s_r\ep_r,\a^*),$$ so $$(s_r\ep_r+s_t\ep_t,\a^*)=2(s_t\ep_t,\a^*)=(s_t\ep_t,s_t\ep_t)=(s_r\ep_r+s_t\ep_t,s_r\ep_r+s_t\ep_t)/2.$$ This means that $$\{s_r\ep_r+s_t\ep_t\mid r,t\in T,\; r\neq t\}\sub\aa_{lg}.$$ Now pick distinct indices $r_1,r_2,r_3\in T.$ Then  $\a:=s_{r_1}\ep_{r_1},\b:=s_{r_2}\ep_{r_2}+s_{r_3}\ep_{r_3}\in\aa,$ but $\a-\b\not\in\rre^i.$ This  contradicts  Proposition \ref{main} and consequently,  $\rre^i$ cannot be of type $B_T.$
\smallskip

Now suppose that $R$ is of type $D_T$ for some    index set $T$ with $|T|\geq4.$ We may assume $R=\{0,\pm(\ep_i\pm\ep_j)\mid i,j\in T;\;i\neq j\}.$
As the subtract of two nonzero orthogonal  roots of $\rre^i$ is not a root, by Proposition \ref{main}, \begin{equation}
\label{type-d}
\parbox{4in}{$\aa$ dose not contain two nonzero orthogonal roots of $\rre^i.$}
\end{equation}
Contemplate Proposition \ref{main} and fix $i_0,j_0\in T$ as well as $r_0,s_0\in\{\pm 1\}$ with $r_0\ep_{i_0}+s_0\ep_{j_0}\in \aa.$ By (\ref{type-d}), Lemma \ref{not-decom} and the fifth condition of a locally finite root supersystem, $(r_0\ep_{i_0}-s_0\ep_{j_0},\a^*)=0$ and so \begin{equation}\label{equal2}(r_0\ep_{i_0},\a^*)=(s_0\ep_{j_0},\a^*).\end{equation}
Next we claim that for each $s\in T\setminus\{i_0\},$ there is $r_s\in\{\pm 1\}$ such that \begin{equation}\label{final5}(r_0\ep_{i_0}+r_s\ep_{s},\a^*)\neq 0\andd(r_0\ep_{i_0}- r_s\ep_{s},\a^*)=0.\end{equation} Indeed, using (\ref{type-d}), Lemma \ref{not-decom} and the fifth condition of a locally finite root supersystem, one can see that it is impossible to have $(r_0\ep_{i_0}+ \ep_{s},\a^*)\neq0$ and $(r_0\ep_{i_0}- \ep_{s},\a^*)\neq0.$ Also if $(r_0\ep_{i_0}\pm \ep_{s},\a^*)= 0,$ we have $(r_0\ep_{i_0},\a^*)=\pm (\ep_{s},\a^*).$ So $(r_0\ep_{i_0},\a^*)=0$ which together with (\ref{equal2}) contradicts the fact that $(r_0\ep_{i_0}+s_0\ep_{j_0},\a^*)\neq 0.$  This completes the proof  of the claim. Next we note that  as $(r_0\ep_{i_0}+r_s\ep_{s},\a^*)\neq 0,$ either  $r_0\ep_{i_0}+r_s\ep_{s}+\a^*\in R$ or $r_0\ep_{i_0}+r_s\ep_{s}-\a^*\in R.$ In the former case, we have using Lemma \ref{not-decom} that $$(r_0\ep_{i_0}+r_s\ep_{s},\a^*)=-(r_0\ep_{i_0}+r_s\ep_{s},r_0\ep_{i_0}+r_s\ep_{s})/2.$$ This together with   (\ref{equal2}), (\ref{final5}) the fact that  $r_0\ep_{i_0}+s_0\ep_{j_0}\in\aa,$ implies that {\small \begin{eqnarray*}
(r_0\ep_{i_0},r_0\ep_{i_0})=(r_0\ep_{i_0}+s_0\ep_{j_0},r_0\ep_{i_0}+s_0\ep_{j_0})/2
&=&(r_0\ep_{i_0}+s_0\ep_{j_0},\a^*)\\&=&2(r_0\ep_{i_0},\a^*)\\
&=&(r_0\ep_{i_0}+r_s\ep_{s},\a^*)\\&=&-(r_0\ep_{i_0}+r_s\ep_{s},r_0\ep_{i_0}+r_s\ep_{s})/2\\
&=&-(r_0\ep_{i_0},r_0\ep_{i_0}).\end{eqnarray*}}

\noindent This makes a contradiction. Therefore $r_0\ep_{i_0}+r_s\ep_{s}-\a^*\in R$ and so by Lemma \ref{not-decom},
\begin{equation}\label{dual}r_0\ep_{i_0}+r_s\ep_{s}\in \aa.\end{equation}
Using the same argument as above, for $t\in  T\setminus\{j_0\},$ one finds $k_t\in\{\pm 1\}$
with $s_0\ep_{j_0}+k_t\ep_t\in \aa.$ This together with (\ref{dual}) implies that    $ \a:=r_0\ep_{i_0}+r_t\ep_{t},\b:=s_0\ep_{j_0}+k_{t'}\ep_{t'}\in \aa$ for distinct elements $t,t'\in T\setminus\{i_0,j_0\},$ but $\a-\b\not\in\rre^i$ contradicting Proposition \ref{main}.

Finally, we assume $\rre^i$ is of type $BC_T$ for a nonempty  index set $T.$ We assume $\rre^i=\{\pm\ep_i,\pm(\ep_i\pm\ep_j)\mid i,j\in T\}.$ If $\aa_{sh}=\aa\cap(\rre^i)_{sh}\neq\emptyset,$ then there is $i_0\in T$ such that $(\ep_{i_0},\a^*)\neq 0.$ Therefore $(2\ep_{i_0},\a^*)\neq 0$ and so by Lemma \ref{not-decom} and the definition of a locally finite root supersystem, there are $r,s\in\{\pm1\}$ such that $2(\ep_{i_0},\a^*)/(\ep_{i_0},\ep_{i_0})=r$ and   $2(2\ep_{i_0},\a^*)/(2\ep_{i_0},2\ep_{i_0})=s.$ But this implies that $s=r/2$ which is a contradiction. So $\aa_{sh}=\emptyset,$ i.e. $(\ep_i,\a^*)=0$ for all $i\in T.$ Therefore $\aa=\emptyset,$ a contradiction.
\smallskip

\noindent{\bf Step 2.} If $|I|=2,$ $\rre^i$ is of type $A:$   We recall from Proposition \ref{main} that $\aa=\{\a\in \rre^i\setminus\{0\}\mid 2(\a,\a^*)=(\a,\a)\}$ is a nonempty set.
Take $j\in I\setminus\{i\},$ then using Proposition \ref{main}, one finds $\b\in \rre^j$ with $(\b,\a^*)=(\b,\b)/2.$ Set $\d:=\a^*-\b\in\rim$ and suppose $\a\in\aa.$ One knows that    $\gamma_\a:=\a^*-\a\in\rim.$ If $(\d,\gamma_\a)\neq0,$ then $R\cap\{\d+\gamma_\a,\d-\gamma_\a\}\neq\emptyset.$ Since $p_*(\d+\gamma_\a)=2\a^*,$ by Lemma \ref{alphastar}, $\d+\gamma_\a\not\in R,$ so $\d-\gamma_\a\in R;$ and using the same lemma, we get that  $\d-\gamma_\a\in \rre.$ This contradicts Corollary \ref{note}.  Therefore $(\d,\gamma_\a)=0.$ So \begin{eqnarray*}
0=(\d,\gamma_\a)=(\a^*-\b,\a^*-\a)&=&-(\a^*,\a)-(\a^*,\b)+(\a,\b)\\&=&-(\a,\a)/2-(\b,\b)/2.\end{eqnarray*} Thus we get  that
$$(\a_1,\a_1)=(\a_2,\a_2);\;\;\a_1,\a_2\in\aa.$$ This together with Step 1 and  Proposition \ref{pre} completes the proof.
\qed
\begin{Pro}
\label{pre-class}
Suppose that  $I\sub\{1,2\}$ is a nonempty set and for  $i\in I,$ $S_i$ is a locally finite root system of type $\dot A_{T_i}$ for an index set $T_i$ with $|T_i|\geq 2.$  Suppose further that $|T_1|\neq |T_2|$ if $I=2$ and $T_1, T_2$ are finite sets. Then up to isomorphism, there is a unique irreducible locally finite root supersystem of imaginary type whose real roots form a locally finite root system isomorphic to  $S:=\op_{i\in I} S_i.$
\end{Pro}
\pf Example \ref{im-sys}(2),(3) guarantees the existence of such locally finite root supersystems. Now suppose $(\v_1,\fm_1,R_1)$ and $(\v_2,\fm_2,R_2)$ are two irreducible locally finite root supersystems of imaginary type with  $(R_1)_{re}$ and $(R_2)_{re}$  isomorphic to  $S=\op_{i\in I} S_i.$ For $i\in I,$ suppose
$S_i=\{\ep^i_t-\ep^i_s\mid t,s\in T_i\}.$
By an identification, we may assume $(R_1)_{re}=(R_2)_{re}=S.$ So there is  a nonzero scalar $r\in\bbbf$ with
\begin{equation}\label{inv}r(u,v)_1=(u,v)_2 \hbox{ for all } u,v\in \hbox{span}_\bbbf S.\end{equation}
Suppose $j=1,2$ and fix $\a^*_j\in(R_j)_{ns}^\times.$ By Proposition \ref{pre}, for $i\in I,$ there is $t_i^j\in T_i$ such that $\{\a\in (R_j)_{re}^\times\mid (\a^*_j,\a)_j\neq 0\}=\cup_{i\in I}\{\pm (\ep^i_{t_i^j}-\ep^i_{t})\mid t\in T_i\setminus\{t_i^j\}\}$ and that   for  $i\in I,$  there is $k^{i}_j\in \{\pm1\}$ such that
$$(\a^*_j,\ep^i_{t_i^j}-\ep^i_{t})_j=-k^{i}_j(\ep^i_{t_i^j}-\ep^i_{t},\ep^i_{t_i^j}-\ep^i_{t})_j/2\andd (\a_j^*,\ep^i_{r}-\ep^i_{t})_j=0;\; t,r\in T\setminus\{t_i^j\}.$$  Now define $\varphi:\v_1\longrightarrow \v_2$ by $$\begin{array}{ll}\a_1^*\mapsto \a_2^*,\;\;
\ep^i_{t_i^1}-\ep^i_{t}\mapsto k^{i}_1k^{i}_2(\ep^i_{t_i^2}-\ep^i_{t}),\;\\\ep^i_{t_i^1}-\ep^i_{t_i^2}\mapsto -k^{i}_1k^{i}_2(\ep^i_{t_i^1}-\ep^i_{t_i^2})\;\;\;\;\;\;\;\;\;\;\;\;\;\;(t\in T\setminus\{t_i^1,t_i^2\},\;i\in I).
\end{array}$$ Then for $i\in I$ and $t,s\in T\setminus\{t_i^1,t_i^2\},$ by (\ref{inv}), we have
\begin{eqnarray*}
(\varphi(\a^*_1),\varphi(\ep^i_{t_i^1}-\ep^i_{t}))_2=
k^{i}_1k^{i}_2(\a^*_2,\ep^i_{t_i^2}-\ep^i_{t})_2
&=&-k^{i}_1k^{i}_2k^{i}_2(\ep^i_{t_i^2}-\ep^i_{t},\ep^i_{t_i^2}-\ep^i_{t})_2/2\\
&=&-k^{i}_1(\ep^i_{t_i^2}-\ep^i_{t},\ep^i_{t_i^2}-\ep^i_{t})_2/2\\
&=&-k^{i}_1(\ep^i_{t_i^1}-\ep^i_{t},\ep^i_{t_i^1}-\ep^i_{t})_2/2\\
&=&-rk^{i}_1(\ep^i_{t_i^1}-\ep^i_{t},\ep^i_{t_i^1}-\ep^i_{t})_1/2\\
&=&r(\a^*_1,\ep^i_{t_i^1}-\ep^i_{t})_1,
\end{eqnarray*}
and
\begin{eqnarray*}
(\varphi(\a^*_1),\varphi(\ep^i_{t_i^1}-\ep^i_{t_i^2}))_2&=&
-k^{i}_1k^{i}_2(\a^*_2,\ep^i_{t_i^1}-\ep^i_{t_i^2})_2\\
&=&k^{i}_1k^{i}_2(\a^*_2,\ep^i_{t_i^2}-\ep^i_{t_i^1})_2\\
&=&-k^{i}_1k^{i}_2k^{i}_2(\ep^i_{t_i^2}-\ep^i_{t_i^1},\ep^i_{t_i^2}-\ep^i_{t_i^1})_2/2\\
&=&-k^{i}_1(\ep^i_{t_i^1}-\ep^i_{t_i^2},\ep^i_{t_i^1}-\ep^i_{t_i^2})_2/2\\
&=&-rk^{i}_1(\ep^i_{t_i^1}-\ep^i_{t_i^2},\ep^i_{t_i^1}-\ep^i_{t_i^2})_1/2\\
&=&r(\a^*_1,\ep^i_{t_i^1}-\ep^i_{t_i^1})_1.
\end{eqnarray*}
Also we have
\begin{eqnarray*}
(\varphi(\ep^i_{t_i^1}-\ep^i_{t}),\varphi(\ep^i_{t_i^1}-\ep^i_{s}))_2&=&
k^{i}_1k^{i}_2k^{i}_1k^{i}_2(\ep^i_{t_i^2}-\ep^i_{t},\ep^i_{t_i^2}-\ep^i_{s})_2\\
&=&(\ep^i_{t_i^2}-\ep^i_{t},\ep^i_{t_i^2}-\ep^i_{s})_2\\
&=&r(\ep^i_{t_i^2}-\ep^i_{t},\ep^i_{t_i^2}-\ep^i_{s})_1
\end{eqnarray*}
and  by Lemma \ref{inv-form}, we get that
\begin{eqnarray*}
(\varphi(\ep^i_{t_i^1}-\ep^i_{t}),\varphi(\ep^i_{t_i^1}-\ep^i_{t_i^2}))_2&=&
-k^{i}_1k^{i}_2k^{i}_1k^{i}_2(\ep^i_{t_i^2}-\ep^i_{t},\ep^i_{t_i^1}-\ep^i_{t_i^2})_2\\
&=&-(\ep^i_{t_i^2}-\ep^i_{t},\ep^i_{t_i^1}-\ep^i_{t_i^2})_2\\
&=&-r(\ep^i_{t_i^2}-\ep^i_{t},\ep^i_{t_i^1}-\ep^i_{t_i^2})_1\\
&=&r(\ep^i_{t_i^2}-\ep^i_{t},\ep^i_{t_i^2}-\ep^i_{t_i^1})_1\\
&=&r(\ep^i_{t_i^1}-\ep^i_{t},\ep^i_{t_i^1}-\ep^i_{t_i^2})_1.
\end{eqnarray*}
But by Proposition \ref{conjugate}($ii$), for $j=1,2,$ $\{\a_j^*,\ep^i_{t_i^j}-\ep^i_{t}\mid i\in I,\; t\in T_i\setminus\{t_i^j\}\}$ is  a basis for $\v_j,$ so we get that $(\varphi(v),\varphi(u))_2=r(u,v)_1$ for all $u,v\in \v_1.$ On the other hand, one can see that $\varphi(r_\a(\b))=r_{\varphi(\a)}\varphi(\b).$ This together with Proposition \ref{conjugate}($ii$) and the fact that the Weyl group of $R_j$ is generated by the reflections based on the elements of $\{\ep^i_{t_i^j}-\ep^i_{t}\mid i\in I,\; t\in T_i\setminus\{t_i^j\}\}$ (see \cite[Lemma 5.1]{LN2}) implies that $\varphi(R_1)=R_2.$
This completes the proof.
\qed
\begin{Pro}
\label{pre-class2}
Suppose that $S$ is a locally finite root system of type $C_T$ for an index set $T$ with $|T|\geq 2.$ Then up to isomorphism, there is a unique irreducible locally finite root supersystem of imaginary type whose real roots form a locally finite root system isomorphic to $ S.$
\end{Pro}
\pf One knows from Example \ref{im-sys}(1) that such locally finite root supersystems  exist. Now suppose $(\v_1,\fm_1,R_1)$ and $(\v_2,\fm_2,R_2)$ are two irreducible locally finite root supersystems of imaginary type whose real roots form locally finite root systems isomorphic to $S.$ Suppose  $S=\{\pm(\ep_r\pm\ep_s)\mid r,s\in T\}.$  Using an identification, we may assume $R_1=R_2=S$ and so there is a nonzero scalar $r\in\bbbf$ such that $$r(u,v)_1=(u,v)_2;\;\;\hbox{for all } u,v\in\hbox{span}_\bbbf S.$$ For $j=1,2,$ fix $\a^*_j\in(R_j)_{ns}^\times.$ By Proposition \ref{pre}, there is  $t_j\in T$ such that $\{\a\in (R_j)_{re}^\times\mid (\a^*_j,\a)\neq 0\}=\{\pm 2\ep_{t_j},\pm (\ep_{t_j}\pm\ep_{t})\mid t\in T\setminus\{t_j\}\}.$ By Proposition \ref{pre} and Lemma \ref{not-decom}($i$), there are $k_j\in \{\pm1\}$ such that $$(\a^*_j,\ep_{t_j})_j=-k_j(\ep_{t_j},\ep_{t_j})_j\hbox{ and } (\a_j^*,\ep_{t})_j=0$$ for all $t\in T\setminus\{t_j\}.$ Now define $\varphi:\v_1\longrightarrow \v_2$ by $$\begin{array}{ll}\a_1^*\mapsto \a_2^*,\;\; \; \ep_{t_1}\mapsto k_1k_2\ep_{t_2},\;\;\; \ep_{t_2}\mapsto k_1k_2\ep_{t_1},\;\;\ep_{t}\mapsto k_1k_2\ep_t;\;t\in T\setminus\{t_1,t_2\}.
\end{array}$$
Now as in the previous Proposition, one can see that $\varphi$ defines an isomorphism between $(\v_1,\fm_1,R_1)$ and $(\v_2,\fm_2,R_2).$\qed

Using Theorem \ref{type}, Propositions \ref{pre-class}, \ref{pre-class2} and Example \ref{im-sys}, we have the following theorem:
\begin{Thm}[Classification Theorem for Imaginary Type]\label{classification I}
Suppose that $T,P$ are index sets with $|T|,|P|\geq2$ and $|T|\neq|P|$ if $T,P$ are both finite. Fix a symbol $\a^*$ and pick $t_0\in T$ and $p_0\in P.$ Consider the free $\bbbf$-module $X:=\bbbf\a^*\op\op_{t\in T}\bbbf\ep_t\op\op_{p\in P}\bbbf\d_p$ and define the symmetric bilinear form $$\fm:X\times X\longrightarrow \bbbf$$ by $$\begin{array}{ll}
(\a^*,\a^*):=0,(\a^*,\ep_{t_0}):=1,(\a^*,\d_{p_0}):=1\\
(\a^*,\ep_t):=0,(\a^*,\d_q):=0&t\in T\setminus\{t_0\},q\in P\setminus\{p_0\}\\
(\ep_t,\d_p):=0,(\ep_t,\ep_s):=\d_{t,s},(\d_p,\d_q):=-\d_{p,q}&t,s\in T,p,q\in P.
\end{array}$$
 Take $R$ to be $\rre\cup \rim^\times$ as in the following table
$${\small
\begin{tabular}{|c|c|c|}
\hline
type &$\rre$&$\rim^\times$
\\\hline
$\dot A(0,T)$& $\{\ep_t-\ep_s\mid t,s\in T\}$&$\pm\w\a^*$\\
\hline
$\dot C(0,T)$& $\{\pm(\ep_t\pm\ep_s)\mid t,s\in T\}$&$\pm\w\a^*$\\
\hline
$\dot A(T,P)$& $\{\ep_t-\ep_s,\d_p-\d_q\mid t,s\in T,p,q\in P\}$&$\pm\w\a^*$\\
\hline
\end{tabular}}$$
in which $\w$ is the subgroup of $GL(X)$ generated by the reflections $r_\a$ $(\a\in \rre\setminus\{0\})$ mapping $\b\in X$ to $\b-\frac{2(\b,\a)}{(\a,\a)}\a,$ then $(\v:=\hbox{span}_\bbbf R,\fm\mid_{\v\times\v}, R)$  is an irreducible locally finite root supersystem of imaginary type and conversely each irreducible locally finite root supersystem of imaginary type  is isomorphic to one and only one of these root supersystems; in particular each irreducible locally finite root supersystem is a direct union of irreducible finite root supersystems of the same type.
\end{Thm}

\subsection{Real type}
In this subsection, we assume that  $R$ is an irreducible locally finite root supersystem of  real type. We recall that  $\{0\}\neq \rre=\op_{i\in I} \rre^i$ is the decomposition of $R$ into nonzero irreducible subsystems and that $\v_{re}^i=\hbox{span}_\bbbf \rre^i.$ So we have $$\v=\sum_{i\in I}\v_{re}^i.$$
For $i\in I,$ take $p_i:\v\longrightarrow \v_{re}^i$ to be the orthogonal projection map on $\v_{re}^i.$ For $\a\in R,$ we define  $\supp(\a):=\{i\in I\mid p_i(\a)\neq0\}$ and call it the {\it support} of $\a.$ We mention that if $\a\in \rre,$ then $|\supp(\a)|=1.$

\begin{Pro}
\label{direct union}
The irreducible locally finite root supersystem $R$ is a direct union of  finite root supersystems.
\end{Pro}
\pf If $\rim=\{0\},$ then $R=\rre$ and so $|I|=1.$ In this case, we are done using \cite[Corollary 3.15]{LN}. Now assume $\rim\neq \{0\}.$ Fix $0\neq \d\in\rim$ and suppose that  $\supp(\d)=\{i_1,\ldots,i_n\}.$ For $1\leq j\leq n,$ take $T_{i_j}\sub \rre^{i_j}$ to be a finite set with $p_{i_j}(\a)\in\hbox{span}_\bbbf T_{i_j}$ and for $i\in I\setminus\{i_1,\ldots,i_n\},$ set $T_i:=\emptyset.$ Now for $i\in I,$ take $\Lam_i$ to be an index set such that $\{(\rre^i)_{\lam_i}\mid \lam_i\in\Lam_i\}$ is the class of all  irreducible full subsystems of $\rre^i$ containing $T_i.$ We know that for $i\in I,$ $\Lam_i$ is a directed set under the ordering  $``\preccurlyeq"$  defined by $\lam\preccurlyeq\mu$ if $(\rre^i)_{\lam}\sub (\rre^i)_{\mu}.$ One knows that $\rre^i$ is the direct union of $\{(\rre^i)_{\lam_i}\mid \lam_i\in\Lam_i\}.$  Set $\Lam:=\Pi_{i\in I}\Lam_i,$ the Cartesian product of $\Lam_i$'s.  For $\lam=(\lam_i)_{i\in I},\mu=(\mu_i)_{i\in I}\in\Lam,$ we say $\lam\preccurlyeq\mu$ if $\lam_i\preccurlyeq\mu_i$ for all $i\in I.$ Next we take $\{I_\gamma\mid \gamma\in\Gamma\},$ where $\Gamma$ is an index set, to be the class of all finite subsets of $I$ containing $\{i_1,\ldots,i_n\}.$ We consider the ordering $``\preccurlyeq"$ on $\Gamma$ defined by $\gamma_1\preccurlyeq\gamma_2$ ($\gamma_1,\gamma_2\in \Gamma$) if $I_{\gamma_1}\sub I_{\gamma_2}.$ For $\lam=(\lam_i)_{i\in I}\in\Lam$ and $\gamma\in \Gamma,$ we set $\v^{(\lam,\gamma)}:=\sum_{i\in I_\gamma}(\v_{re}^i)_{\lam_i}$ in which $(\v_{re}^i)_{\lam_i}:=\hbox{span}_{\bbbf}(\rre^i)_{\lam_i}.$  For pairs $(\lam,\gamma),(\lam',\gamma')\in \Lam\times\Gamma,$ we say $(\lam,\gamma)\preccurlyeq(\lam',\gamma')$ if $\lam\preccurlyeq\lam'$ and $\gamma\preccurlyeq\gamma'.$ Then $\v$ is the direct union of $\{\v^{(\lam,\gamma)}\mid (\lam,\gamma)\in\Lam\times\Gamma\}.$ Now set $$R^{(\lam,\gamma)}:=R\cap \v^{(\lam,\gamma)};\;\;(\lam,\gamma)\in\Lam\times\Gamma.$$ We can see that  $R=\cup_{(\lam,\gamma)\in\Lam\times \Gamma}R^{(\lam,\gamma)}.$ Now we fix $(\lam=(\lam_i)_{i\in I},\gamma)\in\Lam\times \Gamma$ and show that $R^{(\lam,\gamma)}$ is a finite root supersystem in $\v^{(\lam,\gamma)}:$

\begin{itemize}
\item{$R^{(\lam,\gamma)}$ spans $\v^{(\lam,\gamma)}:$} Indeed we have
\begin{eqnarray*}
  \v^{(\lam,\gamma)}=\sum_{i\in I_\gamma}\v_{re}^i=\sum_{i\in I_\gamma}\hbox{span}_{\bbbf}(\rre^i)_{\lam_i}
\sub\hbox{span}_\bbbf R^{(\lam,\gamma)}\sub \v^{(\lam,\gamma)}.
\end{eqnarray*}

\smallskip

\item{The form restricted to $\v^{\lam,\gamma}$ is nondegenerate:} We know from Lemma \ref{inv-form} that  the form $\fm$ restricted to $(\v^i_{re})_{\lam_i},$ $i\in I_\gamma,$ is nondegenerate. Now as $((\v^i_{re})_{\lam_i},(\v^j_{re})_{\lam_j})=\{0\}$ $(i,j\in I_\gamma,\;i\neq j),$ we get that the form restricted to $\v^{\lam,\gamma}$ is nondegenerate.

\smallskip

\item{$R^{(\lam,\gamma)}=-R^{(\lam,\gamma)}$:} It is immediate.

\smallskip

\item{$R^{(\lam,\gamma)}\cap\rre$ is a finite set:} We know that $$R^{(\lam,\gamma)}\cap\rre\sub \rre\cap\v^{(\lam,\gamma)}\sub\rre\cap(\v_{re}\cap\v^{(\lam,\gamma)}). $$
Now as  $\rre$ is locally finite in $\v_{re},$ we get that $\rre\cap(\v_{re}\cap\v^{(\lam,\gamma)})$  is a finite set  and consequently so is $R^{(\lam,\gamma)}\cap\rre.$

\smallskip

\item{  $R^{(\lam,\gamma)}$ is invariant under $r_\a$ for $\a\in \rre^\times\cap R^{(\lam,\gamma)}:$ Suppose  $\d\in R^{(\lam,\gamma)},$ then  $r_\a(\d)\in R\cap\v^{(\lam,\gamma)}$ and so we are done. }

\smallskip

\item{for $\a,\d\in R^{(\lam,\gamma)}$ with $(\a,\a)=0$ and $(\a,\d)\neq 0,$ $\{\a\pm\d\}\cap R^{(\lam,\gamma)}\neq \emptyset:$  It is immediate as $\{\a\pm\d\}\cap R^{(\lam,\gamma)}=\{\a\pm\d\}\cap R.$ }

\smallskip

\item{$R^{(\lam,\gamma)}$ is finite:} The above items imply that $$(\v^{(\lam,\gamma)},\fm\mid_{_{\v^{(\lam,\gamma)}\times \v^{(\lam,\gamma)}}},R^{(\lam,\gamma)})$$ is a locally finite root supersystem whose real roots form a finite root system. Now Lemma \ref{finite} implies that $R^{(\lam,\gamma)}$ is finite.
\end{itemize}
This completes the proof.\qed

\begin{lem}
\label{serg3}
(i) For $\a\in R^\times,$ $\supp(\a)\neq\emptyset.$

(ii) If $\a,\b\in \rim$ are such that either $(\a,\b)\neq0$ or $\supp(\a)\cap \supp(\b)\neq\emptyset,$ then either  $\supp(\a)\sub\supp(\b)$ or $\supp(\b)\sub\supp(\a).$
\end{lem}
\pf Using
Proposition \ref{direct union} together with \cite[Lemmas 2.1, 2.2 and Corollary 2.4]{serg}, we get the result.\qed

\medskip
For $\a,\b\in \rim\setminus\{0\},$ we say $\a,\b$ are {\it equivalent} and write $\a\sim\b$ if there is $\gamma\in\rim$ such that $\supp(\a)\cup\supp(\b)\sub\supp(\gamma).$ Using Lemma \ref{serg3}, one can see that this defines an equivalence relation on $\rim\setminus\{0\}.$ 
So $\rim\setminus\{0\}$ is the disjoint union of equivalence classes $\rim^k,$ where $k$  runs over an index set $K.$ Setting $S_k:=\cup_{\a\in\rim^k}\supp(\a)\sub I,$ we have the following lemma:

\begin{lem}\label{serg4}
(i) If $k,k'\in K$ with $k\neq k',$ then $S_k\cap S_{k'}=\emptyset.$

(ii) If $k\in K$ and $i\in I$ are such that $i\not\in S_k,$ then $(\rim^k,\rre^i)=\{0\}.$

(iii) If $k,k'\in K$ with $k\neq k',$ then $(\rim^k,\rim^{k'})=\{0\}.$

(iv) $|K|=1$ and if $K=\{k\},$ then $S_k=I.$
\end{lem}
\pf
$(i)$ Suppose that $k,k'$ are distinct elements of $K$ and  $i\in S_k\cap S_{k'}.$ So there are $\a\in \rim^k$ and $\b\in\rim^{k'}$ such that $i\in\supp(\a)\cap\supp(\b).$ Now using Lemma \ref{serg3}, we get either  $\supp(\a)\sub\supp(\b)$ or $\supp(\b)\sub\supp(\a).$ This implies that  $\a\sim\b,$ a contradiction.

$(ii),(iii)$
It is trivial.
%
%

$(iv)$ If $\a\in \rre$  and $k\in K,$ then for $\b\in \rim^k,$ using  Lemma \ref{serg3}($i$), $\supp(r_\a(\b))\cap \supp(\b)\neq \emptyset,$ so by Lemma \ref{serg3}($ii$), we have  $r_\a(\b)\sim\b.$ Now using this together with part ($iii$) and Lemma \ref{general}, we get that $|K|=1.$ Next, we suppose that $K=\{k\}$ and  show that $I=S_k.$ To the contrary, suppose that $I\neq S_k,$ then we have  $R^\times=(\cup_{i\in I\setminus S_k}\rre^i\setminus\{0\})\uplus (\cup_{i\in S_k}\rre^i\setminus\{0\}\cup\rim^\times)$ with  $(\cup_{i\in I\setminus S_k}\rre^i\setminus\{0\}, (\cup_{i\in I}\rre^i\setminus\{0\}\cup\rim^\times))=\{0\}$ which contradicts the irreducibility of $R.$ So $I=S_k.$\qed

\begin{Pro}
\label{direct union-irr}
The irreducible  locally finite root supersystem $R$ is a direct union of irreducible finite root supersystems.
\end{Pro}
\pf If $\rim=\{0\},$ we get the result using \cite[Corollary 3.15]{LN}. So suppose $\rim\neq\{0\}$ and  fix $\d\in \rim.$ Using the same notation as in Proposition \ref{direct union},
it is enough to show that for each pair $(\lam,\gamma)\in \Lam\times\Gamma,$
there is $(\lam',\gamma')\in \Lam\times\Gamma$ with $(\lam,\gamma)\preccurlyeq(\lam',\gamma')$
such that $R^{(\lam',\gamma')}$ is irreducible.

Let $\lam=(\lam_i)_{i\in I}\in\Lam$ and $\gamma\in\Gamma.$ We know that $\v^{(\lam,\gamma)}=\sum_{i\in I_\gamma}(\v_{re}^i)_{\lam_i}.$ By Lemma \ref{serg4}($iv$), for each $i\in I_\gamma,$ there is $\d_i\in \rim$ with $i\in \supp(\d_i).$ Again using Lemma \ref{serg4}($iv$), one finds $\d_0\in\rim$ with $\cup_{i\in I_\gamma}\supp(\d_i)\sub\supp(\d_0).$ Suppose that $\supp(\d_0)=\{i_1,\ldots,i_t\}$ and note that $I_\gamma\sub\supp(\d_0).$
For $1\leq j\leq t,$ there is a finite set $S_j\sub \rre^{i_j}$ with $p_{i_j}(\d_0)\in\hbox{span}_\bbbf(S_j).$ For $1\leq j\leq t,$ take $\mu_j\in\Lam_{i_j}$ to be such that $S_j\cup (\rre^{i_j})^{\lam_{i_j}}\sub (\rre^{i_j})^{\mu_j}.$ Also, take $\gamma'$ to be such that $I_{\gamma'}=\{i_1,\ldots,i_t\}$ and define $\lam'=(\lam'_i)_{i\in I}\in\Lam$ with $\lam'_{i_j}:=\mu_j$ and $\lam'_i=\lam_i$ for $1\leq j\leq t$ and $i\not\in I\setminus I_{\gamma'}.$ We see that $(\lam,\gamma)\preccurlyeq(\lam',\gamma').$ Also $\d_0\in R^{\lam',\gamma'}$ and \begin{equation}\label{full-supp}(\d_0,(\rre^{i})^{\lam'_i})\neq \{0\}\;\;\;\; (i\in I_{\gamma'}).\end{equation} Now  for each $\a\in R^{(\lam',\gamma')}\setminus\{0\},$  there is $i\in I_{\gamma'}$ with $p_i(\a)\neq 0,$ in particular $(\a,(\rre^{i})^{\lam'_i})\neq \{0\}.$ Using (\ref{full-supp}) together with  the fact that $(\rre^{i})^{\lam'_i}$ is connected for all $i\in I_{\gamma'},$ we get that $\d_0$ is connected to all nonzero roots of $R^{(\lam',\gamma')}.$ This in turn implies that $R^{(\lam',\gamma')}$ is irreducible; see Lemma \ref{decom}.
\qed
\begin{Pro}
\label{serg5}
Recall that $R$ is an irreducible  locally finite root supersystem of real type and use the same notation as above, then  we have  $|I|\leq 3.$ Also if  $\a\in \rim^\times,$ then $supp(\a)=I.$
\end{Pro}
\pf It follows using Proposition \ref{direct union-irr} and \cite[Proposition 2.6 and Lemma 2.7]{serg}.

Using the above proposition, we assume that $I=\{1,\ldots, n\}$ for some positive integer  $n\leq 3.$


\begin{rem}\label{rem2}
{\rm Using the same arguments as in Propositions \ref{direct union} and \ref{direct union-irr}, we know that if $\rim\neq \{0\},$ then $R$ is a direct union of irreducible finite root supersystems $R^{(\lam,\gamma)},$ where $(\lam,\gamma)$ runs over a subset  $X$ of $\Lam\times\Gamma.$ Using Lemma \ref{serg5}, $I=I_\gamma=I_{\gamma'}$ for $(\lam,\gamma),(\lam',\gamma')\in X.$ So there is a subset $\Lam'$ of $\Lam$ such that  $R$ is a direct union of irreducible finite root supersystems  $R^{\lam}$ $(\lam\in\Lam')$ in which $R^{\lam}:=R\cap \sum_{i\in I}(\v^i_{re})_{\lam_i}$ for  $\lam=(\lam_1,\ldots,\lam_n).$
Without loss of generality, we always assume for $i\in I,$ $\cap_{(\lam_1,\ldots,\lam_n)\in\Lam'} (\rre^i)_{\lam_i}\neq \{0\}.$ In particular, for $1\leq i\leq n,$ if $\rre^i$ is of one of types $B_T,C_T,D_T$ or $BC_T$ as in (\ref{locally-finite}), for some  index set $T$ with $|T|\geq 2,$ we fix two distinct elements of $T$ and call them 1,2 and  assume $\cap_{(\lam_1,\ldots,\lam_n)\in\Lam'} (\rre^i)_{\lam_i}\cap\hbox{span}_\bbbf\{\ep_1,\ep_2\}\neq \{0\}.$}
\end{rem}
\begin{lem}
\label{small-orbit}
Suppose that $i\in I,$ then $p_i(\rim)$ is a union of small orbits. In particular, if $\rre^i$ is of type $A,$ it is of finite rank.
\end{lem}
\pf Using the same argument as in \cite[Proposition 3.5]{serg} and considering Propositions \ref{serg5} and \ref{smallorbits}, we are done.\qed

\smallskip

In the following theorem, using the classification of finite root supersystems \cite[\S 6 and Theorem 5.10]{serg} together with Proposition \ref{smallorbits} and Remark \ref{rem3}(ii), we give the classification of  irreducible locally finite root supersystems of real type.

\begin{Thm}
\label{classification II}
 Suppose that $2\leq n\leq 3$ and  $S_1,\ldots,S_n$  are irreducible locally finite root systems in $\u_1,\ldots,\u_n,$ respectively. Consider the locally finite root system  $S:=S_1\op \cdots\op S_n$ in $\v:=\u_1\op\cdots\op\u_n.$ Take $\w$ to be the Weyl group of $S.$ Fix a symmetric invariant bilinear form $\fm_i$ on $\u_i.$  With the same notation as in the text,  for $1\leq i\leq n,$ if $S_i$ is a finite root system of rank $\ell\geq 2,$ by $\{\omega_1^i,\ldots,\omega_\ell^i\},$ we mean a  set of fundamental weights for $S_i$ and if $S_i$ is one of locally finite root systems $B_T, C_T, D_T$ or $BC_T$ of infinite rank as in (\ref{locally-finite}), by $\omega_1^i,$ we mean $\ep_1,$ where $1$ is a distinguished element of $T.$ Also if $S_i$ is one of  the  finite root systems $\{0,\pm\a\}$ of type $A_1$ or $\{0,\pm\a,\pm2\a\}$ of type $BC_1,$ we set   $\omega_1^i:=\a/2.$
Consider $\d^*$ and $R:=\rre\cup\rim^\times$ as in the following table:
$${\tiny
\begin{tabular}{|c|l|c|c|c|c|}
\hline
$ n$& $ S_i\;(1\leq i\leq n)$&$ R_{re}$&$\d^*$&$ R_{ns}^\times$&\hbox{type}\\
\hline
$ 2$& $ S_1=A_\ell,\; S_2=A_\ell$ $(\ell\in\bbbz^{\geq1})$&$S_1\op S_2$&$\omega_1^1+\omega_1^2$&$\pm\w\d^*$&$A(\ell,\ell)$\\
\hline
$2$& $ S_1=B_T,\;S_2=BC_{T'}$ $(|T|,|T'|\geq2)$&$S_1\op S_2$&$\omega_1^1+\omega_1^2$&$\w\d^*$&$B(T,T')$\\
\hline
$2$& $ S_1=BC_T,\;S_2=BC_{T'}$ $(|T|,|T'|>1)$&$S_1\op S_2$&$\omega_1^1+\omega_1^2$&$\w\d^*$&$BC(T,T')$\\
\hline
$2$& $ S_1=BC_T,\;S_2=BC_{T'}$ $(|T|=1,|T'|=1)$&$S_1\op S_2$&$2\omega_1^1+2\omega_1^2$&$\w\d^*$&$BC(T,T')$\\
\hline
$2$& $ S_1=BC_T,\;S_2=BC_{T'}$ $(|T|=1,|T'|>1)$&$S_1\op S_2$&$2\omega_1^1+\omega_1^2$&$\w\d^*$&$BC(T,T')$\\
\hline
$2$& $ S_1=D_T,\;S_2=C_{T'}$ $(|T|\geq3,|T'|\geq2)$&$S_1\op S_2$&$\omega_1^1+\omega_1^2$&$\w\d^*$&$D(T,T')$\\
\hline
$2$& $ S_1=C_T,\;S_2=C_{T'}$ $(|T|,|T'|\geq2)$&$S_1\op S_2$&$\omega_1^1+\omega_1^2$&$\w\d^*$&$C(T,T')$\\
\hline
$2$& $ S_1=A_1,\;S_2=BC_{T}$ $(|T|=1)$&$S_1\op S_2$&$2\omega_1^1+2\omega_1^2$&$\w\d^*$&$B(1,T)$\\
\hline
$2$& $ S_1=A_1,\;S_2=BC_{T}$ $(|T|>1)$&$S_1\op S_2$&$2\omega_1^1+\omega_1^2$&$\w\d^*$&$B(1,T)$\\
\hline
$2$& $ S_1=A_1,\;S_2=C_T$ $(|T|\geq 2)$&$S_1\op S_2$&$\omega_1^1+\omega_1^2$&$\w\d^*$&$C(1,T)$\\
\hline
$2$& $ S_1=A_1,\;S_2=B_3$ &$S_1\op S_2$&$\omega_1^1+\omega_3^2$&$\w\d^*$&$AB(1,3)$\\
\hline
$2$& $ S_1=A_1,\;S_2=D_{T}$ $(|T|\geq3)$&$S_1\op S_2$&$\omega_1^1+\omega_1^2$&$\w\d^*$&$D(1,T)$\\
\hline
$2$& $ S_1=BC_1,\;S_2=B_T$ $(|T|\geq2)$&$S_1\op S_2$&$2\omega_1^1+\omega_1^2$&$\w\d^*$&$B(T,1)$\\
\hline
$2$& $ S_1=BC_1,\;S_2=G_2$ &$S_1\op S_2$&$2\omega_1^1+\omega_1^2$&$\w\d^*$&$G(1,2)$\\
\hline
$3$& $ S_1=A_1,\;S_2=A_1,\; S_3=A_1$ &$S_1\op S_2\op S_3$&$\omega_1^1+\omega_1^2+\omega_1^3$&$\w\d^*$&$D(2,1,\lam)$\\
\hline
$3$& $ S_1=A_1,\; S_2= A_1,\; S_3:=C_T$ $(|T|\geq2)$&$S_1\op S_2\op S_3$&$\omega_1^1+\omega_1^2+\omega_1^3$&$\w\d^*$&$D(2,T)$\\
\hline
\end{tabular}
}$$
Normalize the form $\fm_i$ on $\u_i$ such that  $(\d^*,\d^*)=0$ and that for type $D(2,T),$ $(\omega_1^1,\omega_1^1)_1=(\omega_1^2,\omega_1^2)_2.$  Then for $\fm:=\op_{i=1}^n\fm_i,$ $(\v,\fm,R)$ is an irreducible locally finite root supersystem of real type and conversely if  $(\v,\fm,R)$ is an irreducible locally finite root supersystem of real type, it is either an irreducible locally finite root system or  isomorphic to one and only  one of the locally finite root supersystems listed in the above table.
\end{Thm}

\centerline{Acknowledgment}

 This research was in part
supported by a grant from IPM (No. 91170415) and partially carried out in IPM-Isfahan branch.

\end{document}